%% file: LongRange_Revised.tex
\documentclass[11pt, reqno]{article}
\input{header}

\input{commands}

\usepackage[margin=1.025in]{geometry}
\usepackage{tikz}
\usepackage{pgfplots}
\usepackage{extarrows}
\usepackage{titling}
\usepackage{commath}
\usepackage{wasysym}
\usepackage{mathtools}
\usepackage{titlesec}
\newcommand{\eps}{\varepsilon}
\usepackage{bbm}
\usepackage{setspace}
\usepackage{enumitem}
\usepackage{booktabs}

\usepackage[textsize=tiny]{todonotes}

\usepackage{cite}

\newcommand{\Aut}{\operatorname{Aut}}

\newcommand{\bP}{\mathbf P}
\newcommand{\bE}{\mathbf E}

\def\P{\mathbb{P}}

\newcommand{\sS}{\mathscr{S}}

\DeclareMathSymbol{\leqslant}{\mathalpha}{AMSa}{"36} 
\DeclareMathSymbol{\geqslant}{\mathalpha}{AMSa}{"3E} 
\DeclareMathSymbol{\eset}{\mathalpha}{AMSb}{"3F}     

\renewcommand{\epsilon}{\varepsilon}

\newcommand{\bp}{\mathbf{p}}

\makeatletter
\pgfdeclareshape{slash underlined}
{
  \inheritsavedanchors[from=rectangle] 
  \inheritanchorborder[from=rectangle]
  \inheritanchor[from=rectangle]{north}
  \inheritanchor[from=rectangle]{north west}
  \inheritanchor[from=rectangle]{north east}
  \inheritanchor[from=rectangle]{center}
  \inheritanchor[from=rectangle]{west}
  \inheritanchor[from=rectangle]{east}
  \inheritanchor[from=rectangle]{mid}
  \inheritanchor[from=rectangle]{mid west}
  \inheritanchor[from=rectangle]{mid east}
  \inheritanchor[from=rectangle]{base}
  \inheritanchor[from=rectangle]{base west}
  \inheritanchor[from=rectangle]{base east}
  \inheritanchor[from=rectangle]{south}
  \inheritanchor[from=rectangle]{south west}
  \inheritanchor[from=rectangle]{south east}
  \inheritanchorborder[from=rectangle]
  \foregroundpath{
    \southwest \pgf@xa=\pgf@x \pgf@ya=\pgf@y
    \northeast \pgf@xb=\pgf@x \pgf@yb=\pgf@y
    \pgf@xc=\pgf@xa
    \advance\pgf@xc by .5\pgf@xb
    \pgf@yc=\pgf@ya
    \advance\pgf@xc by -1.3pt
    \advance\pgf@yc by -1.8pt
    \pgfpathmoveto{\pgfqpoint{\pgf@xc}{\pgf@yc}}
    \advance\pgf@xc by  2.6pt
    \advance\pgf@yc by  3.6pt
    \pgfpathlineto{\pgfqpoint{\pgf@xc}{\pgf@yc}}
    \pgfpathmoveto{\pgfqpoint{\pgf@xa}{\pgf@ya}}
    \pgfpathlineto{\pgfqpoint{\pgf@xb}{\pgf@ya}}
 }
}
\tikzset{nomorepostaction/.code=\let\tikz@postactions\pgfutil@empty}

\makeatother

\pretitle{\begin{flushleft}\Large}
\posttitle{\par\end{flushleft}\vskip 0.5em
}
\preauthor{\begin{flushleft}}
\postauthor{
\\
\vspace{0.5em}
\footnotesize{Statslab, DPMMS, University of Cambridge\\ Email: 
\href{mailto:t.hutchcroft@maths.cam.ac.uk}{t.hutchcroft@maths.cam.ac.uk}}
\end{flushleft}}
\predate{\begin{flushleft}}
\postdate{\par\end{flushleft}}
\title{\bf Power-law bounds for critical long-range percolation below the upper-critical dimension}

\renewenvironment{abstract}
 {\par\noindent\textbf{\abstractname.}\ \ignorespaces}
 {\par\medskip}

\titleformat{\section}
  {\normalfont\large \bf}{\thesection}{0.4em}{}

\author{{\bf Tom Hutchcroft}}

\begin{document}

\date{\small{\today}}

\maketitle

\setstretch{1.1}

\vspace{-1.35em}
\noindent \emph{Dedicated to Harry Kesten, November 19, 1931 – March 29, 2019}
\vspace{1.35em}

\begin{abstract} We study long-range Bernoulli percolation on $\Z^d$ in which each two vertices $x$ and $y$ are connected by an edge with probability $1-\exp(-\beta \|x-y\|^{-d-\alpha})$. 
 It is a  theorem of Noam Berger (\emph{Commun.\ Math.\ Phys.}, 2002) that if $0<\alpha<d$ then there is no infinite cluster at the critical parameter $\beta_c$. We give a new, quantitative proof of this theorem establishing the power-law upper bound
\[
\bP_{\beta_c}\bigl(|K|\geq n\bigr) \leq C n^{-(d-\alpha)/(2d+\alpha)}
\]
for every $n\geq 1$, where $K$ is the cluster of the origin. We believe that this is the first rigorous power-law upper bound for a Bernoulli percolation model that is neither planar nor expected to exhibit mean-field critical behaviour.

As part of the proof, we establish a universal inequality implying that the maximum size of a cluster in percolation on any finite graph is of the same order as its mean with high probability. We apply this inequality to derive a new rigorous hyperscaling inequality $(2-\eta)(\delta+1)\leq d(\delta-1)$ 
relating the cluster-volume exponent $\delta$ and two-point function exponent $\eta$.
\end{abstract}

\section{Introduction}
\label{sec:intro}

Let $d\geq 1$ and suppose that $J:\Z^d \to [0,\infty)$ is both \textbf{symmetric} in the sense that $J(x)=J(-x)$ for every $x\in \Z^d$ and \textbf{integrable} in the sense that $\sum_{x\in \Z^d} J(x) <\infty$. For each $\beta\geq 0$, \textbf{long-range percolation} on $\Z^d$ with intensity $J$ is the random graph with vertex set $\Z^d$ in which we choose whether or not to include each potential edge $\{x,y\}$ independently at random with inclusion probability $1-\exp(-\beta J(y-x))$. Note that this model is equivalent to nearest-neighbour percolation when $J(x)=\mathbbm{1}(\|x\|_1=1)$. Here we will instead be most interested in the case that $J(x)$ decays like an inverse power of $\|x\|$, so that
\begin{equation}
\label{eq:Jassumption}
J(x) \sim A\|x\|^{-d-\alpha} \qquad \text{ as $x\to\infty$}
\end{equation}
for some constants $A>0$ and $\alpha >0$. 
We denote the law of the resulting random graph by $\bP_{\beta}=\bP_{J,\beta}$ and refer to the connected components of this random graph as \textbf{clusters}.
Studying the geometry of these clusters leads to many interesting questions, some of which are motivated by applications to modeling `small-world' phenomena in physics, epidemiology, the social sciences, and so on; see e.g.\ \cite[Section 1.4]{MR2094435} and \cite[Section 10.6]{MR1864966} for background and many references. Although substantial progress on these questions has been made over the last forty years, with highlights of the literature including \cite{MR868738,MR3306002,MR3997479,ding2013distances,MR2663638,MR1913075,MR2094435,MR2430773}, many further important problems remain open.  

In this paper we study the \emph{phase transition} in long-range percolation. Given $d\geq 1$ and a symmetric, integrable function $J:\Z^d \to [0,\infty)$,
 we define the \textbf{critical parameter}
\[ \beta_c = \beta_c(J) = \sup\bigl\{ \beta \geq 0 : \bP_\beta \text{ is supported on configurations with no infinite clusters}\bigr\}.\]
Elementary path-counting arguments yield that there are no infinite clusters almost surely when $\beta \sum_x J(x) <1$, and hence that $\beta_c \geq 1/\sum_x J(x) >0$ under the assumption that $J$ is locally finite.
  When $d =1$ and $J$ is of the form \eqref{eq:Jassumption}, the model has a non-trivial phase transition in the sense that $0<\beta_c<\infty$ if and only if $\alpha \leq 1$, while for $d\geq 2$ the phase transition is non-trivial for every $\alpha>0$ \cite{schulman1983long,newman1986one}. As with nearest-neighbour percolation, the model is expected to exhibit many interesting fractal-like features when $\beta=\beta_c$ (see e.g.\ \cite{ding2013distances,MR3306002,MR2430773}), but proving this rigorously seems to be a very difficult problem in general. 

It is a surprising fact that our understanding of long-range percolation models is  better than our understanding of their nearest-neighbour counterparts in many situations. Indeed, it is a remarkable theorem of Noam Berger \cite{MR1896880} that long-range percolation on $\Z^d$ undergoes a continuous phase transition  in the sense that there is no infinite cluster at $\beta_c$ whenever $d \geq 1$ and $0<\alpha <d$. The corresponding statement for nearest-neighbour percolation with $d\geq 2$ is of course a notorious open problem needing little further introduction. 
While it is widely believed that the phase transition should be continuous for all $\alpha>0$ and $d\geq 2$, it is a theorem of Aizenman and Newman \cite{MR868738} that the model undergoes a \emph{discontinuous} phase transition  when $d=\alpha=1$, so that the condition $\alpha<d$ cannot be removed from Berger's result in general.

Berger's proof works by showing that the set of $\beta$ for which an infinite cluster exists a.s.\ is open, and gives little quantitative control of percolation at the critical parameter $\beta_c$ itself.
In this paper we give a new, quantitative proof of Berger's result that yields an explicit power-law upper bound on the tail of the volume of the cluster of the origin at criticality under the same assumptions. We write $K_0$ for the cluster of the origin, write $\Lambda_r=[-r,r]^d \cap \Z^d$ for each $r\geq 0$, and write $\{x \leftrightarrow y\}$ for the event that $x$ and $y$ belong to the same cluster.
 
\begin{theorem}
\label{thm:main}
Let $d\geq 1$, let $J:\Z^d \to (0,\infty)$ be symmetric and integrable, and suppose that there exists $\alpha < d$, $c>0$, and $r_0<\infty$ such that $J(x)\geq c \|x\|_1^{-d-\alpha}$ for every $x\in \Z^d$ with $\|x\|_1 \geq r_0$. Then there exists a constant $C$ such that
\begin{align}
\bP_{\beta}(|K_0| \geq n) &\leq C n^{-(d-\alpha)/(2d+\alpha)} \\ \hspace{-3cm} \text{and} \hspace{2cm}
 \frac{1}{|\Lambda_r|}
  \sum_{x\in \Lambda_r} \bP_\beta(0 \leftrightarrow x) &\leq C r^{-2(d-\alpha)/(3d)}
\end{align}
for every $\beta \leq \beta_c$, $n \geq 1$, and $r\geq 1$. In particular, there are almost surely no infinite clusters at the critical parameter $\beta_c$.
\end{theorem}

The theorem is most interesting when $d < 6$ and $\alpha > d/3$, in which case the model is \emph{not} expected to have mean-field behaviour and high-dimensional techniques such as the lace expansion \cite{MR1043524,MR3306002,MR2430773} should not apply. Indeed, we believe that \cref{thm:main} is the first rigorous, non-trivial power-law upper bound for a critical Bernoulli percolation model that is neither two-dimensional nor expected to be described by mean-field critical exponents. 



Let us now discuss interpretations of our results in terms of critical exponents. It is strongly believed that the large-scale behaviour of critical (long-range or nearest-neighbour) percolation on $d$-dimensional Euclidean lattices is described by \emph{critical exponents} \cite[Chapters 9 and 10]{grimmett2010percolation}. The most relevant of these exponents to us are  traditionally denoted $\delta$ and $\eta$ and are believed to describe the distribution of the cluster of the origin at criticality via the asymptotics
\begin{align*}
\bP_{\beta_c}(|K_0| \geq n) &\approx n^{-1/\delta} &\text{ as }n&\to\infty\\
\hspace{-2cm}\text{and} \hspace{2cm} \bP_{\beta_c}(x\leftrightarrow y) &\approx \|x-y\|^{-d+2-\eta} &\text{ as }\|x-y\|&\to\infty,
\end{align*}
where $\approx$ means that the ratio of the logarithms of the two sides tends to $1$ in the relevant limit. These exponents are expected to depend on the dimension $d$ and the long-range parameter $\alpha$ (if appropriate) but not on the small-scale details of the model such as the choice of lattice. It is an open problem of central importance to  prove the existence of and/or compute these exponents, as well as to prove that they are universal in this sense. 
Significant progress has been made in high dimensions 
($d >6$ or $\alpha < d/3$)
 \cite{MR1043524,MR762034,MR1127713,fitzner2015nearest,MR3306002,MR1959796,MR2430773}, where it is known that $\delta=2$ and $\eta=0$ for several large classes of examples, and for nearest-neighbour models in two dimensions \cite{MR879034,smirnov2001critical,smirnov2001critical2,lawler2002one}, where it has been proven in particular that $\delta=91/5$ and $\eta=5/24$ for site percolation on the triangular lattice as predicted by Nienhuis \cite{nienhuis1987coulomb}. Important partial progress for other two-dimensional planar lattices has been made by Kesten \cite{MR633715,MR894549,MR879034} and Kesten and Zhang \cite{MR893131}. Progress in intermediate dimensions has however been extremely limited. \cref{thm:main} can be seen as a modest first step towards understanding the problem in this regime, and implies that for long-range percolation with $0<\alpha<d$ the exponents $\delta$ and $\eta$ satisfy
\begin{equation}
\delta \leq \frac{2d+\alpha}{d-\alpha} \qquad \text{ and } \qquad 2-\eta \leq \frac{1}{3}d + \frac{2}{3}\alpha
\label{eq:mainexponents}
\end{equation}
whenever they are well-defined.  (Conversely, the mean-field lower bound of Aizenman and Barsky~\cite{aizenman1987sharpness} implies that $\delta$ satisfies $\delta \geq 2$ whenever it is well-defined; see also \cite[Proposition 1.3]{duminil2015new}.) See \cref{subsec:physics} for a discussion of how these bounds compare to the non-rigorous predicted values of $\eta$ and $\delta$ in the physics literature.
 We remark that similar bounds on other exponents including the susceptibility exponent $\gamma$, gap exponent $\Delta$, and cluster density exponent $\beta$ can be obtained from \eqref{eq:mainexponents} using the rigorous scaling inequalities $\gamma \leq \delta-1$, $\Delta \leq \delta$, and $\beta \geq 2/\delta$ proven in \cite{1901.10363} and \cite{MR912497}.

\medskip

\textbf{Hyperscaling inequalities.} As a part of our proof, we also prove a new rigorous \emph{hyperscaling inequality} $(2-\eta)(\delta+1)\leq d(\delta-1)$ for both long-range and nearest-neighbour percolation. To prove this inequality, we first prove a universal inequality implying in particular that the maximum cluster size in percolation on any finite graph is of the same order as its mean with high probability. Both results are of independent interest, and are discussed in detail in \cref{sec:hyperscaling}.


\medskip

\textbf{Other graphs.} Our methods are not very specific to the hypercubic lattice $\Z^d$, and can also be used to establish very similar results for long-range percolation on, say, arbitrary transitive graphs of $d$-dimensional volume growth. We now formulate an even more general version of our theorem, which will follow by essentially the same proof. The definitions introduced here will also be used throughout the rest of the paper. 
Given a graph $G$ and a vertex $v$, we write $E^\rightarrow_v$ for the set of oriented edges emanating from $v$. (We will often abuse notation by identifying this set with the corresponding set of unoriented edges.)
We define a \textbf{weighted graph} $G=(V,E,J)$ to be a countable graph $(V,E)$ together with an assignment of positive \textbf{weights} 
 $\{J_e : e \in E\}$ such that $\sum_{e\in E^\rightarrow_v} J_e < \infty$ for each $v\in V$. Locally finite graphs can be considered as weighted graphs by setting $J_e \equiv 1$. A graph automorphism of $(V,E)$ is a weighted graph automorphism of $(V,E,J)$ if it preserves the weights, and a weighted graph $G$ is said to be \textbf{transitive} if for every two vertices $x$ and $y$ in $G$ there exists an automorphism of $G$ sending $x$ to $y$. We say that a weighted graph is \textbf{simple} if there is at most one edge between any two vertices. Given a weighted graph $G=(V,E,J)$ and $\beta\geq 0$, we define Bernoulli-$\beta$ bond percolation on $G$ to be the random subgraph of $G$ in which each edge is chosen to be either retained or deleted independently at random with retention probability $1-e^{-\beta J_e}$, and write $\bP_\beta=\bP_{G,\beta}$ for the law of this random subgraph.

\begin{theorem}
\label{thm:general}
Let $G=(V,E,J)$ be an infinite, simple, unimodular transitive weighted graph, let $o$ be a vertex of $G$, and suppose that 
%
there exist constants $1/2 < a < 1$, $c>0$, and $\eps_0>0$ such that $|\{ e\in E^\rightarrow_o : J_e \geq \eps\}| \geq c \eps^{-a}$ for every $0<\eps \leq \eps_0$. Then $ \beta_c < \infty$ and there exists a constant $C$ such that
\[
\bP_{\beta}(|K_o| \geq n) \leq C n^{-(2a-1)/(a+1)}
\]
for every $0\leq \beta \leq \beta_c$ and $n\geq 1$. In particular, there are almost surely no infinite clusters at the critical parameter $\beta_c$.
\end{theorem}

The hypothesis of \emph{unimodularity} is a technical condition that holds in most natural examples, including all amenable transitive weighted graphs and all weighted graphs defined in terms of a countable group $\Gamma$ and a symmetric, integrable function $J:\Gamma \to [0,\infty)$ by $V=\Gamma$, $E=\{ \{g,h\} : g,h\in \Gamma, J(g^{-1} h)>0\}$, and $J(\{g,h\})=J(g^{-1}h)$ for each $\{g,h\}\in E$ \cite{MR1082868}.  (As in the case of $\Z^d$, we say that a function $J:\Gamma \to [0,\infty)$ on a countable group $\Gamma$ is symmetric if $J(\gamma)=J(\gamma^{-1})$ for every $\gamma \in \Gamma$ and integrable if $\sum_{\gamma \in \Gamma}J(\gamma)<\infty$.) It follows in particular that \cref{thm:general} implies \cref{thm:main}. 
 See \cite[Chapter 8]{LP:book} for further background on unimodularity.

\begin{remark}
\cref{thm:general} also leads to a new proof of a recent theorem of Xiang and Zou \cite{MR4104986} which states that every countably infinite (but not necessarily finitely generated) group $\Gamma$ admits a symmetric, integrable function $J:\Gamma \to [0,\infty)$ for which the associated weighted graph has a non-trivial percolation phase transition. To deduce their theorem from ours, simply pick a bijection $\sigma:\Gamma\to \{1,2,\ldots\}$, let $1<\alpha <2$, and consider the symmetric, integrable function on $\Gamma$ defined by $J(\gamma)= \sigma(\gamma)^{-\alpha}+\sigma(\gamma^{-1})^{-\alpha}$ for every $\gamma \in \Gamma$: the associated long-range percolation model has $\beta_c<\infty$ by \cref{thm:general}. 
We remark also that  Xiang and Zou's proof relied on the results of Duminil-Copin, Goswami, Raoufi, Severo, and Yadin \cite{1806.07733} in the case that the group is finitely generated, while our proof is self-contained. 
It would be interesting if a new proof of the results of \cite{1806.07733} could be derived from \cref{thm:general} by comparison of short- and long-range percolation. 
\end{remark}

\subsection{About the proof}

We now outline the basic structure of our proof and discuss how it compares to previous approaches to critical percolation. 
We begin with a brief overview of the two main strategies that have been employed in the study of critical percolation, which we term the \emph{supercritical strategy} and the \emph{subcritical strategy}.
Broadly speaking, the supercritical strategy has found more success in low-dimensional settings while the subcritical strategy has found more success in high-dimensional settings, but there are notable exceptions in both cases. We write $\theta(p)=\bP_p(|K_o|=\infty)$ for the probability that the origin lies in an infinite cluster.

\medskip

\noindent \textbf{The supercritical strategy.} In this strategy, one attempts to prove that the set $\{p:\theta(p)>0\}$ is open by analysis of percolation under the assumption that $\theta(p)>0$. For example, one may hope to show that if infinite clusters exist then each such cluster $K$ must be `large' in some coarse sense that is strong enough to ensure that $p_c(K)<1$. This approach has been successfully followed both in Berger's analysis of long-range percolation on $\Z^d$ \cite{MR1896880} and in Benjamini, Lyons, Peres, and Schramm's proof that critical percolation on any nonamenable Cayley graph has no infinite clusters \cite{BLPS99b}. Harris's classical proof that $\theta(1/2)=0$ for the square lattice \cite{harris1960lower} can also be thought of in similar terms. Alternatively, one may instead attempt to find a \emph{finite-size characterisation of supercriticality}, that is, a sequence of events $(\cE_n)_{n\geq 1}$ each depending on at most finitely many edges and a sequence of positive numbers $(\delta_n)_{n\geq 0}$ such that
\[
\theta(p)>0 \iff \text{there exists $n\geq 1$ such that $\bP_p(\sE_n) > 1- \delta_n$}
\]
for every $p\in [0,1]$; the existence of such a finite-size characterisation of supercriticality is easily seen to imply that the set $\{p:\theta(p)>0\}$ is open as required. Such finite-size characterisations are typically derived via a renormalization argument, and this strategy often amounts to an alternative formalization of the more geometric strategy discussed above. Successful realisations of this approach include Barsky, Grimmett, and Newman's  analysis \cite{MR1124831,MR1144091} of  half-spaces and orthants in $\Z^d$  and Duminil-Copin, Sidoravicius, and Tassion's analysis  \cite{MR3503025} of two-dimensional slabs $\Z^2 \times [0,r]^{k}$. A popular approach to critical percolation on $\Z^3$ seeks to implement this strategy by eliminating the  `sprinkling' from the proof of the Grimmett-Marstrand theorem \cite{MR1068308}; while this has not yet been done successfully, interesting partial progress in this direction has been made by Cerf \cite{MR3395466}.

Arguments following the supercritical strategy tend to be ineffective in the sense that they give little or no quantitative information about percolation at $p_c$;  see however the recent work of Duminil-Copin, Kozma, and Tassion \cite{duminil2019upper} for some progress towards reversing this trend.  

\medskip

\noindent \textbf{The subcritical strategy.} In this strategy, one attempts to prove that the set $\{p : \theta(p)=0\}$ is closed by proving that some non-trivial upper bound on the distribution of the cluster of the origin holds uniformly throughout the subcritical phase. In contrast to the supercritical strategy, the subcritical strategy is inherently quantitative in nature and typically yields explicit estimates on the distribution of the cluster of the origin at criticality. 
The simplest example of such an argument is the proof that there is no percolation at criticality on any amenable transitive graph of exponential volume growth \cite{Hutchcroft2016944}, which uses elementary subadditivity considerations to prove the uniform bound
\[
\min \{ \bP_p(x \leftrightarrow y) : d(x,y) \leq n\} \leq \operatorname{gr}(G)^{-n}
\]
for every $n\geq 1$ and $p<p_c$, where $\operatorname{gr}(G)=\limsup_{n\to\infty} |B(x,n)|^{1/n}$ is the rate of exponential volume growth of $G$. Left-continuity of connection probabilities then implies that the same bound continues to hold at $p_c$, from which the theorem is easily deduced. 

More sophisticated versions of the subcritical strategy often involve a `bootstrapping' or `forbidden zone' argument.  Such an argument was first used to analyze high-dimensional statistical mechanics models by  Slade  \cite{slade1987diffusion}. To implement such an argument, one aims to prove that some well-chosen estimate, called the \textbf{bootstrapping hypothesis}, \emph{implies a strictly stronger version of itself}. Once this is done, it is usually straightforward to conclude via a continuity argument that the strong form of the estimate holds uniformly throughout the subcritical phase. 
For example, the \emph{lace expansion} for high-dimensional percolation \cite{MR1283177,MR1959796,fitzner2015nearest,MR2430773} works roughly by showing that if $d$ is sufficiently large and $\sG$ denotes the Greens function on $\Z^d$ then for each $p\in [0,p_c)$ we have the implication
\begin{multline}
\label{eq:lace_expansion}
\left(\bP_p(x\leftrightarrow y) \leq 3\, \sG(x,y) \text{ for every $x,y\in \Z^d$}\right)\\ \Rightarrow \left(\bP_p(x\leftrightarrow y) \leq 2\, \sG(x,y) \text{ for every $x,y\in \Z^d$}\right).
\end{multline}
The estimate $\bP_p(x\leftrightarrow y) \leq 3\, \sG(x,y)$ holds trivially when $p$ is small. Since we also have that $\limsup_{x \to \infty} \bP_p(0 \leftrightarrow x)/\sG(0,x) =0$ for every $p<p_c$ by sharpness of the phase transition \cite{aizenman1987sharpness,duminil2015new}, it follows by an elementary continuity argument that $\bP_p(x\leftrightarrow y) \leq 2\, \sG(x,y)$ for every $0 \leq p \leq p_c$ and hence that there is no infinite cluster at $p_c$ as desired. (In fact the bootstrapping hypothesis used in the lace expansion analysis of  percolation is more complicated than this, but the essence of the argument is as described.)  See \cite{heydenreich2015progress,MR2239599} for an overview of this method and \cite{slade2020simple,MR3765883} for recent work simplifying the implementation of the lace expansion for weakly self-avoiding walk.

In this paper we build upon a new version of the subcritical strategy that has been developed in our recent works \cite{1808.08940,HermonHutchcroftIntermediate,Hutchcroft2020Ising}.  The most basic form of the method was first used to prove power-law upper bounds for percolation on groups of exponential growth in \cite{1808.08940}, while a more sophisticated version of the method, closer to that employed here, was subsequently used to analyze critical percolation on certain groups of stretched-exponential volume growth in joint work with Hermon~\cite{HermonHutchcroftIntermediate}. Very recently, similar ideas have also been used to prove continuity of the phase transition for the Ising model on nonamenable groups \cite{Hutchcroft2020Ising}.

 Let us now outline how this method works. In \cite{1808.08940}, we built upon the work on Aizenman, Kesten, and Newman \cite{MR901151} to prove an upper bound on the probability of a certain two-arm-type event, which we called the \emph{two-ghost inequality}, that holds universally for all unimodular transitive graphs. One formulation of this inequality states that if $G=(V,E)$ is a connected, locally finite, transitive unimodular graph (e.g.\ $G=\Z^d$) and  $\sS_{e,n}$ denotes the event that the endpoints of the edge $e$ are in distinct clusters each of which touches (i.e., contains a vertex incident to) at least $n$ edges and at least one of which is finite, then
\begin{equation}
\label{eq:twoghostintro}
\sum_{e\in E^\rightarrow_o} \bP_p(\sS_{e,n}) \leq 66 \deg(o) \sqrt{\frac{1-p}{p n}} 
\end{equation}
for every $p\in (0,1]$, $n\geq 1$, and $o \in V$.
 An extension of the two-ghost inequality to long-range models (including certain dependent models) was proven in \cite[Section 3]{Hutchcroft2020Ising}, which we give a further improvement to in \cref{thm:two_ghost_S}. The two-ghost inequality can sometimes be used to prove that the percolation phase transition is continuous via the following rough strategy, which we implement a version of in this paper:
 \begin{enumerate}
\item Assume as a bootstrapping hypothesis some well-chosen upper bound $\bP_\beta(|K_o|\geq n) \leq h(n)$ for each $n\geq 1$ with $h(n)\to0$ as $n\to\infty$ and that holds trivially when $\beta$ is very small and that decays subexponentially, so that $\lim_{n\to\infty} h(n)^{-1} \bP_\beta(|K_o|\geq n) =0$ for every $\beta<\beta_c$ by sharpness of the phase transition \cite{duminil2015new,1901.10363}. Choosing which bound to use is a potentially subtle matter which may involve trial and error.  Heuristically, there is a `Goldilocks principle' that needs to be satisfied when choosing the bootstrapping hypothesis appropriately: A bound that is too weak will be of too little use as an input to proceed further into the argument, while a bound that is too strong will be too difficult to re-derive in a stronger form as required for the bootstrapping argument to come full circle. In particular, any bound decaying faster than $n^{-1/2}$ cannot possibly work. In this paper we are able to consider power-law upper bounds as seems most natural, while in \cite{HermonHutchcroftSupercritical} the optimal upper bound making the argument work was of the form $C e^{-\log^\eps n}$ for small $\eps>0$.
\item Find some way to convert the bootstrapping hypothesis $\bP_\beta(|K_o|\geq n)\leq h(n)$ into a two-point function upper bound $\bP_\beta(o \leftrightarrow x)\leq f(x)$ for some function $f$ that hopefully decays reasonably quickly as $x\to\infty$ for at least some well-chosen choices of $x$. In \cite{HermonHutchcroftIntermediate}, for example, this is done by letting $X$ be a random walk and bounding $\bP_\beta(o \leftrightarrow X_k)$ via spectral techniques. Here we will instead prove such a bound using hyperscaling inequalities as discussed in \cref{sec:hyperscaling}.
\item Use the Harris-FKG inequality and a union bound to observe that $\bP_\beta(\sS_{o,x,n}') \geq \bP_\beta(|K_o|\geq n)^2- \bP_\beta(o\leftrightarrow x)$, where $\sS_{o,x,n}'$ is the event that $o$ and $x$ belong to distinct clusters of size at least $n$, then prove an upper bound of the form $\bP_\beta(\sS_{o,x,n}') \leq F(x) \bP_\beta(\sS_{e,n}')$ for some appropriately chosen edge $e=e(x)$ and some function $F(x)$ that is hopefully not too large. In \cite{1808.08940,HermonHutchcroftIntermediate} this second step is done via a surgery argument using the finite-energy property of percolation. In our setting this step is much simpler and more efficient since we can just take $e$ to be the `long edge' connecting $o$ to $x$ and take $F(x)\equiv 1$.
\item Put steps 2 and 3 together to get an inequality of the form
\[
\bP_\beta(|K_o|\geq n) \leq \sqrt{ F(x) \bP_\beta(\sS'_{e,n})+f(x)}
\]
for every $n\geq 1$ and every vertex $x$ under the assumption that $\beta<\beta_c$ and that the bootstrapping hypothesis holds. The proof will work if bounding $\bP_\beta(\sS'_{e,n})$ using the two-ghost inequality and optimizing over the choice of $x$ leads to a bound $\bP_\beta(|K_o|\geq n) \leq g(n)$ that is a strict improvement of the bootstrapping hypothesis in the sense that $g(n)<h(n)$ whenever $h(n) < 1$. (The function $g$ must not depend on the choice of $0\leq \beta < \beta_c$.) Once this has been done successfully, it follows by an elementary continuity argument (using that $\lim_{n\to\infty} h(n)^{-1} \bP_\beta(|K_o|\geq n) =0$ for every $\beta<\beta_c$) that the bound $\bP_\beta(|K_o|\geq n) \leq g(n)$ holds for all $0\leq \beta \leq \beta_c$ and $n\geq 1$, and hence that  there is no percolation at criticality as desired.
\end{enumerate}

\subsection{A short proof of a weaker result}
\label{subsec:short_proof}

In order to give a simple illustration of how the strategy sketched above can be applied to long-range percolation on $\Z^d$, we now give a quick proof of a weaker result requiring $\alpha<d/4$ rather than $\alpha<d$ and giving a worse upper bound on the exponent $\delta$.

\begin{prop}
\label{prop:weak}
Let $d\geq 1$, let $J:\Z^d \to [0,\infty)$ be symmetric and integrable, and suppose that there exists $\alpha < d/4$, $c>0$, and $r_0<\infty$ such that $J(x)\geq c \| x \|_1^{-d-\alpha}$ for every $x\in \Z^d$ with $\|x\|_1 \geq r_0$. Then there exists a constant $C$ such that
\begin{align*}
\bP_{\beta}(|K_0| \geq n) &\leq C n^{-(d-4\alpha)/(4d)}
\end{align*}
for every $0\leq \beta \leq \beta_c$ and $n \geq 1$. In particular, there are almost surely no infinite clusters at the critical parameter $\beta_c$.
\end{prop}

The proof will apply the following special case of the two-ghost inequality of \cite[Corollary 3.2]{Hutchcroft2020Ising}. We will prove a stronger version of this inequality in \cref{sec:twoghost}. For each $e\in E$ and $\lambda >0$, we define $\sS_{e,\lambda}$ to be the event that the endpoints of $e$ are in distinct clusters each of which touches a set of edges of total weight at least $\lambda$ and at least one of which contains only finitely many vertices.

\begin{thm}
\label{thm:twoghostfromIsing} Let $G=(V,E,J)$ be a connected, unimodular, transitive weighted graph, let $o$ be a vertex of $G$, and let $\beta \geq 0$.
Then
\begin{align}
\sum_{e\in E^\rightarrow_o} \sqrt{J_e(e^{\beta J_e}-1)} \bP_{\beta}(\sS_{e,\lambda})  \leq \frac{42}{\sqrt{\lambda}} \qquad \text{ for every $\lambda >0$.}
\end{align}
\end{thm}

\begin{proof}[Proof of \cref{prop:weak}]
By rescaling if necessary, we may assume without loss of generality that $\sum_{x \in \Z^d} J(x) =1$, so that $\beta_c \geq 1$. Let $\theta=(d-4\alpha)/4d<1/4$.
We claim that there exists a constant $C\geq 1$ such that the following implication holds for each $1/2 \leq \beta < \beta_c$ and $1 \leq A < \infty$:
\begin{equation}
\text{$\Bigl(\bP_\beta(|K_0|\geq n) \leq A n^{-\theta}$ for every $n\geq 1\Bigr)$} \\\Rightarrow  \text{$\Bigl(\bP_\beta(|K_0|\geq n) \leq C A^{1/2} n^{-\theta}$ for every $n\geq 1\Bigr)$}.
\label{eq:simple_bootstrap_claim}
\end{equation}

Indeed, fix one such $1/2 \leq \beta <\beta_c$ and suppose that $1 \leq A <\infty$ is such that $\bP_\beta(|K_0|\geq n) \leq A n^{-\theta}$ for every $n\geq 1$. All the constants appearing in the remainder of the proof will be allowed to depend on $d,\alpha,$, $c$, and $r_0$, but not on the choice of $1\leq A < \infty$ or $1/2 \leq \beta<\beta_c$.
 For each $x\in \Z^d$ and $n\geq 1$, let $\sS_{x,n}'$ be the event that $0$ and $x$ belong to distinct clusters each of which contain at least $n$ vertices. Both clusters are automatically finite since $\beta < \beta_c$.
It follows from \cref{thm:twoghostfromIsing} that there exists a constant $C_1$ such that
\[
\sum_{x \in \Z^d} J(x)^{1/2}(e^{\beta J(x)}-1)^{1/2} \bP_\beta(\sS_{x,n}') \leq C_1 n^{-1/2}
\]
for every $n \geq 1$. For each $r\geq r_0$, define $\Lambda_r'=\Lambda_r \setminus \Lambda_{r_0-1}$; we will prove bounds depending on both $n$ and $r$ before optimizing over the choice of $r$ later in the proof. Using the inequality $e^x-1 \geq x$ and the assumption that $J(x) \geq c \|x\|_1^{-d-\alpha}$ for every $x \in \Z^d \setminus \Lambda_{r_0-1}$, it follows that there exists a constant $C_2$ such that
\begin{multline}
\label{eq:twoghost_simpleproof}
\sum_{x\in \Lambda_r'} \bP_\beta(\sS_{x,n}') \leq \max\left\{ J(x)^{-1/2}(e^{\beta J(x)}-1)^{-1/2} : x\in \Lambda_r'\right\} \sum_{x \in \Z^d} J(x)^{1/2}(e^{\beta J(x)}-1)^{1/2} \bP_\beta(\sS_{x,n}') \\
\leq \max\left\{ \beta^{-1/2} J(x)^{-1} : x\in \Lambda_r'\right\}  C_1 n^{-1/2}
\leq C_2 r^{d+\alpha} n^{-1/2}
\end{multline}
for every $n\geq 1$ and $r\geq r_0$.
 On the other hand, we have trivially that there exists a constant $C_3$ such that
\begin{multline}
\label{eq:primitive0}
\sum_{x \in \Lambda_r'}\bP_\beta(0 \leftrightarrow x) = \bE_\beta |K_0 \cap \Lambda_r'| \leq \bE_\beta \left[|K_0| \wedge |\Lambda_r'| \right] = \sum_{n=1}^{|\Lambda_r'|} \bP_\beta(|K_0| \geq n) \\\leq A\sum_{n=1}^{|\Lambda_r'|} n^{-\theta} \leq  C_3 A r^{d(1-\theta)}
\end{multline}
for every $r\geq r_0$. 
We now apply these two bounds to obtain a new bound on $\bP_\beta(|K_0|\geq n)$. We have by a union bound and the Harris-FKG inequality that
\[
\bP_\beta(\sS_{x,n}') \geq \bP_\beta(|K_0|\geq n,|K_x|\geq n)-\bP_\beta(0 \leftrightarrow x) \geq \bP_\beta(|K_0|\geq n)^2-\bP_\beta(0 \leftrightarrow x)
\]
for each $x\in \Z^d$ and $n\geq 1$. Rearranging yields $\bP_\beta(|K_0|\geq n)^2 \leq \bP_\beta(0 \leftrightarrow x) + \bP_\beta(\sS_{x,n}') $ for every  $x\in \Z^d$ and $n\geq 1$, and it follows by  averaging over $x\in \Lambda_r'$ that there exists a constant $C_4$ such that
\begin{equation*}
\bP_\beta(|K_0|\geq n)^2 \leq \frac{1}{|\Lambda_r'|}\sum_{x\in \Lambda_r'} \bP_\beta(0 \leftrightarrow x) + \frac{1}{|\Lambda_r'|}\sum_{x\in \Lambda_r'}\bP_\beta(\sS_{x,n}')
\leq C_4 A r^{-d\theta} + C_4 r^{\alpha} n^{-1/2}
\end{equation*}
for every $r\geq r_0$ and $n\geq 1$, where we applied \eqref{eq:primitive0} and \eqref{eq:twoghost_simpleproof} in the second inequality.
Taking 
$r=r_0 \vee \left\lceil n^{(1-4\theta)/(2\alpha)}\right \rceil$ yields that there exists a constant $C_5$ such that
\begin{align}
\bP_\beta(|K_0|\geq n)^2 
&\leq C_5 A n^{-d \theta(1-4\theta)/(2\alpha)} +  C_5 n^{-2\theta} = C_5 (A+1) n^{-2\theta} \leq 2 C_5 A n^{-2\theta}
\label{eq:simple_bootstrap_final}
\end{align}
for every $n \geq 1$, where we used that $\theta=(d-4\alpha)/4d$ in the central equality. (We arrived at this value of $\theta$ by getting to this stage of the calculation with $\theta$ indeterminate and solving for the value of $\theta$ that made the two powers of $n$ equal.) The inequality \eqref{eq:simple_bootstrap_final} implies the claimed implication \eqref{eq:simple_bootstrap_claim} by taking square roots on both sides.

We now apply the bootstrapping implication \eqref{eq:simple_bootstrap_claim} to complete the proof of the proposition.
For each $1/2 \leq \beta < \beta_c$, we have by sharpness of the phase transition \cite{aizenman1987sharpness,duminil2015new} that $|K_0|$ has finite mean (indeed, it has an exponential tail), and in particular that there exists $1 \leq A < \infty$ such that $\bP_\beta(|K_0|\geq n) \leq A n^{-\theta}$ for every $n\geq 1$. For each $1/2\leq \beta < \beta_c$ we may therefore define
\[
A_\beta = \min\bigl\{1 \leq A < \infty : \bP_\beta(|K_0|\geq n) \leq A n^{-\theta} \text{ for every $n\geq 1$}\bigr\} < \infty.
\]
Since the set we are minimizing over is closed, we have that $\bP_\beta(|K_0|\geq n) \leq A_\beta n^{-\theta}$ for every $n\geq 1$ and $1/2 \leq \beta < \beta_c$.
Moreover, \eqref{eq:simple_bootstrap_claim} implies that there exists a constant $C$ such that 
$A_\beta \leq C A_\beta^{1/2}$
for every $0 \leq \beta < \beta_c$, and since $A_\beta$ is finite for every $1/2\leq \beta <\beta_c$ we may safely rearrange this inequality to obtain that $A_\beta \leq C^2$ for every $1/2\leq \beta <\beta_c$. Thus, we have proven that
$\bP_\beta(|K_0|\geq n) \leq C^{2} n^{-\theta}$
for every $0 \leq \beta < \beta_c$. Considering the standard monotone coupling of $\bP_\beta$ and $\bP_{\beta'}$ for $\beta \leq \beta'$ and taking limits, it follows that the same estimate holds for all $0\leq \beta\leq\beta_c$ and $n\geq 1$ as claimed.
\end{proof}

In order to prove \cref{thm:main}, we will improve the above proof in two ways: In \cref{sec:hyperscaling} we develop a better method to convert volume-tail bounds into two-point function bounds than the primitive method used in \eqref{eq:primitive0}, while in \cref{sec:twoghost} we prove an improved form of the two-ghost inequality that gives better bounds on $\bP_\beta(\sS_{e,n})$ in the case that $e$ is a typical `long' edge. (Each improvement can be used in isolation to prove a result of intermediate strength requiring $\alpha<d/2$.)

\subsection{Comparison to physics predictions}
\label{subsec:physics}

\pgfplotsset{width=0.4\textwidth}
\pgfplotsset{compat=newest}

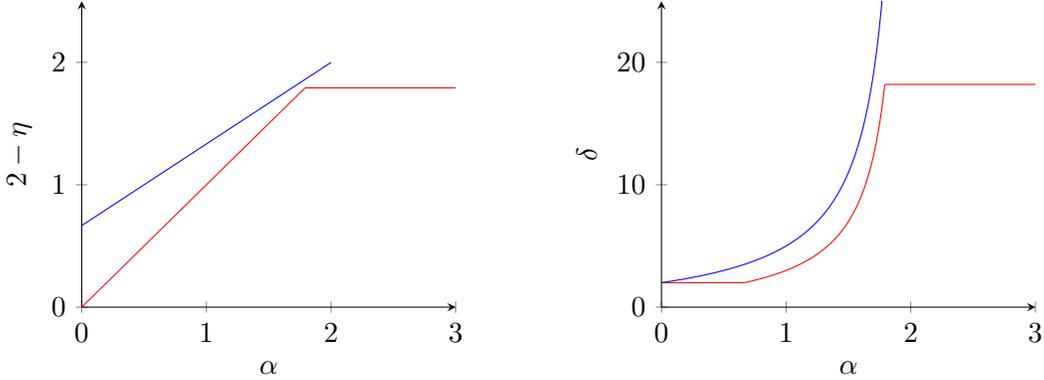
\begin{figure}[t]
\centering
\begin{tikzpicture}
\begin{axis}[
    axis lines = left,
    xlabel = $\alpha$,
    ylabel = {$2-\eta$},
    xmin=0,xmax=3,
    ymax=2.5
]
\addplot [
    domain=0:43/24, 
    samples=100, 
    color=red,
]
{x};
\addplot [
    domain=43/24:3, 
    samples=100, 
    color=red,
]
{43/24};
\addplot [
    domain=0:2, 
    samples=100, 
    color=blue,
    ]
    {(2/3)+(2*x)/3};
\end{axis}
\end{tikzpicture}
    \hspace{1cm}
    \begin{tikzpicture}
\begin{axis}[
    axis lines = left,
    xlabel = $\alpha$,
    ylabel = {$\delta$},
    xmin=0, xmax=3,
    ymin=0, ymax=25
]
\addplot [
    domain=0:2/3, 
    samples=100, 
    color=red,
]
{2};
\addplot [
    domain=2/3:43/24, 
    samples=100, 
    color=red,
]
{(2+x)/(2-x)};
\addplot [
    domain=43/24:3, 
    samples=100, 
    color=red,
]
{91/5};
\addplot [
    domain=0:1.8, 
    samples=100, 
    color=blue,
    ]
    {(4+x)/(2-x)};
\end{axis}
\end{tikzpicture}
    \caption{Our upper bounds (blue) vs.\ the conjectured true values (red) of $2-\eta$ and $\delta$ when $d=2$. }
    \label{fig:2d}
    \vspace{-1em}
\end{figure}

We now give a brief heuristic discussion of how our exponent bounds compare to the values predicted in the physics literature. 
Building upon the work of Sak \cite{sak1973recursion} on long-range Ising models (see also e.g.\ \cite{behan2017scaling,luijten1997interaction}), physicists including Brezin, Parisi, and  Ricci-Tersenghi \cite{brezin2014crossover}  have argued  
 that
  if $\eta(d,\alpha)$ and $\delta(d,\alpha)$ denote the values of the exponents $\eta$ and $\delta$ for long-range percolation in dimension $d$ with long-range parameter $\alpha$ and $\eta_\mathrm{SR}(d)$ and $\delta_\mathrm{SR}(d)$ denote the corresponding nearest-neighbour exponents then
\begin{equation}
\label{eq:eta_prediction}
2-\eta(d,\alpha) = \begin{cases} \alpha & \alpha \leq 2-\eta_\mathrm{SR}(d)\\
2-\eta_\mathrm{SR}(d) & \alpha>2-\eta_\mathrm{SR}(d),
\end{cases}
\end{equation}
with logarithmic corrections to scaling at the `crossover' value $\alpha^*(d) = 2-\eta_\mathrm{SR}(d)$. 
In particular, the exponent $\eta(d,\alpha)$ is predicted to `stick' to its mean-field value of $2-\alpha$ in the interval $(d/3,\alpha^*]$, even though other exponents such as $\delta$ are \emph{not} expected to take their mean-field values in this interval.
Numerical work supporting these predictions has recently been carried out in \cite{gori2017one}.
See \cite{MR3306002,MR4032873,MR2430773} for rigorous proofs in certain high-dimensional cases and \cite{MR3772040,MR3723429} for related rigorous results for the long-range spin $O(n)$ model. Assuming further that $\delta(d,\alpha)$ takes its mean-field value of $2$ when $\alpha \leq d/3$ and that the hyperscaling relation $(2-\eta)(\delta+1)=d(\delta-1)$ is satisfied when $\alpha \geq d/3$ yields that
\vspace{-0.25em}
\begin{equation}
\label{eq:delta_prediction}
\delta(d,\alpha) = 
\begin{cases}
2 & \hspace{1.855em}0< \alpha \leq d/3\\
(d+\alpha)/(d-\alpha) & \hspace{0.84em}d/3 \leq \alpha \leq \alpha^*(d)\\
\delta_\mathrm{SR}(d) & \alpha^*(d) \leq \alpha <\infty.
\end{cases}
\end{equation}
 As discussed above, it is strongly expected and known in some cases that $\eta_\mathrm{SR}(2)=5/24$ and that $\eta_\mathrm{SR}(d)=0$ when $d\geq 6$. On the other hand, it is believed that $\eta_\mathrm{SR}$ takes small negative values for $d\in \{3,4,5\}$: Both numerical estimates \cite{lorenz1998precise,xun2020precise,tiggemann2001simulation,wang2013bond} and non-rigorous renormalization group methods \cite{gracey2015four} give values ranging between $-0.1$ and $-0.01$ in all three cases. (See the Wikipedia page \url{https://en.wikipedia.org/wiki/Percolation_critical_exponents} for a summary.) As such, it is believed that $\alpha^*(d)<d$ for every $d\geq 2$ and hence that that the models treated by \cref{thm:main} should include examples in the same universality class as nearest-neighbour Bernoulli bond percolation on each lattice of dimension $d\geq 2$. (Proving such a universality claim would, however, require a vastly better understanding of these models than that provided by \cref{thm:main}.) The bounds we obtain on the exponents for these models   are of reasonable order, with our upper bounds on $\delta(d,\alpha)$ always within a factor of $2$ of the predicted true values when $\alpha \leq \alpha^*(d) =2-\eta_\mathrm{SR}(d)$. See \cref{fig:2d,fig:3d} for side-by-side comparisons in two and three dimensions.

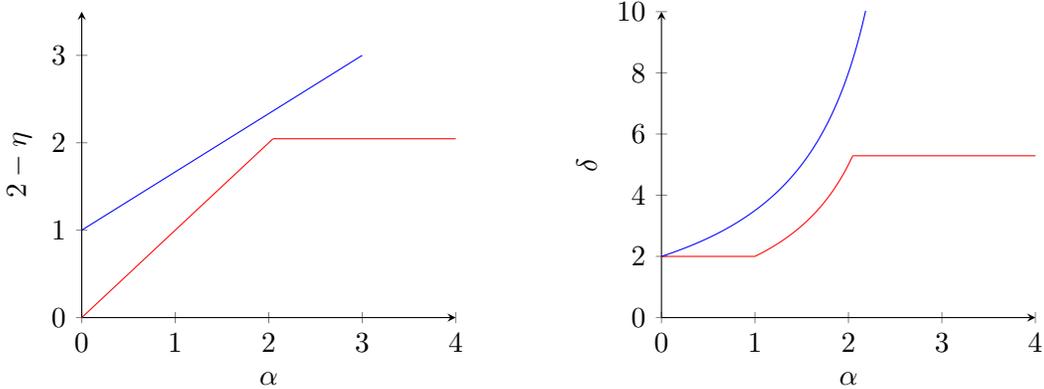
\begin{figure}
\centering
\begin{tikzpicture}
\begin{axis}[
    axis lines = left,
    xlabel = $\alpha$,
    ylabel = {$2-\eta$},
    ymax=3.5
]
\addplot [
    domain=0:2.0457, 
    samples=100, 
    color=red,
]
{x};
\addplot [
    domain=2.0457:4, 
    samples=100, 
    color=red,
]
{2.0457};
\addplot [
    domain=0:3, 
    samples=100, 
    color=blue,
    ]
    {1+(2*x)/3};
\end{axis}
\end{tikzpicture}
    \hspace{1cm}
        \begin{tikzpicture}
\begin{axis}[
    axis lines = left,
    xlabel = $\alpha$,
    ylabel = {$\delta$},
    xmin=0, xmax=4,
    ymin=0, ymax=10
]
\addplot [
    domain=0:1, 
    samples=100, 
    color=red,
]
{2};
\addplot [
    domain=1:2.0457, 
    samples=100, 
    color=red,
]
{(3+x)/(3-x)};
\addplot [
    domain=2.0457:4, 
    samples=100, 
    color=red,
]
{5.2886};
\addplot [
    domain=0:2.5, 
    samples=100, 
    color=blue,
    ]
    {(6+x)/(3-x)};
\end{axis}
\end{tikzpicture}
    \caption{Our upper bounds (blue) vs.\ the conjectured true values (red) of $2-\eta$ and $\delta$ when $d=3$. 
Here we use the numerical values $\alpha^*(3)=2-\eta_\mathrm{SR}(3)\approx 2.0457$ and $\delta_\mathrm{SR}(3)\approx 5.2886$ obtained by applying the scaling and hyperscaling relations to the numerical estimates on the exponents $\nu$ and $\beta/\nu$ obtained by Wang et al.\ in \cite{wang2013bond}. When $\alpha=2.0457\approx \alpha^*(3)$ our upper bound on $\delta$ is about $8.43$.}
\label{fig:3d}
\end{figure}

\section{Hyperscaling inequalities and the maximum cluster size in a box}
\label{sec:hyperscaling}

The proof of \cref{prop:weak} made use of the fact that if Bernoulli bond percolation on some weighted graph $G=(V,E,J)$ satisfies a bound of the form $\sup_{v\in V}\bP_\beta(|K_v|\geq n) \leq A n^{-\theta}$ for some $0\leq \theta <1$ and $A<\infty$ then we have that
\begin{equation}
\label{eq:primitive}
\sum_{v\in \Lambda}\bP_\beta(u \leftrightarrow v) = \bE_\beta |K_u \cap \Lambda| \leq \bE_\beta \left[|K_u| \wedge |\Lambda| \right] \leq A\sum_{n=1}^{|\Lambda|} n^{-\theta} \leq  C(\theta)A |\Lambda|^{1-\theta}
\end{equation}
for every $\Lambda \subseteq V$ and $u\in V$, where $C(\theta)=O(1/(1-\theta))$ depends only on $\theta$. Tasaki \cite{MR926203} observed that this inequality, which holds for arbitrary random graph models on $\Z^d$, can be thought of as giving a primitive \emph{hyperscaling inequality} $(2-\eta)\delta \leq d(\delta-1)$. In this section, we prove an inequality implying the stronger hyperscaling inequality $(2-\eta)(\delta+1) \leq d(\delta-1)$. Note that while the arguments in the rest of the paper can all be applied to certain \emph{dependent} percolation models including the random-cluster model with only a little extra work, the arguments in this section rely on the BK inequality in an essential way and are therefore very specific to Bernoulli percolation.

Let us now briefly review what is known about scaling and hyperscaling relations for Bernoulli percolation. In addition to the critical exponents $\delta$ and $\eta$ that we have already introduced, it is also believed that there exist exponents $\gamma, \Delta,$ $\rho$, and $\beta$ such that 
\begin{align*}
  \bE_{\beta_c-\eps}\left[|K_0|^k \right] &\approx \eps^{-\gamma - \Delta(k-1)} && \text{ as $\eps \downarrow 0$ for each $k\geq 1$}\\
\bP_{\beta_c}(0 \leftrightarrow \partial [-r,r]^d) &\approx r^{-1/\rho} && \text{ as $r \uparrow\infty$, and}\\
\bP_{\beta_c+\eps}(|K_0|=\infty) &\approx \eps^\beta && \text{ as $\eps \downarrow 0$.}
\end{align*}
As before, $\approx$ means that the ratio of the logarithms of the two sides tends to $1$ in the relevant limit. A further critical exponent $\nu$ is expected to describe the \emph{correlation length} $\xi(\beta)$ through the asymptotics $\xi(\beta_c-\eps) \approx \eps^{-\nu}$ as $\eps \downarrow 0$. Intuitively the correlation length is the scale on which off-critical behaviour begins to manifest itself, see \cite[Section 6.2]{grimmett2010percolation} for a precise definition in the nearest-neighbour context.  Heuristic scaling theory predicts that these exponents always satisfy the \emph{scaling relations}
\begin{equation}
\label{eq:scaling}
\gamma = \beta(\delta-1), \qquad \beta \delta = \Delta, \qquad \text{ and } \qquad \gamma = \nu (2-\eta).
\end{equation}
 Below the upper critical dimension, two additional relations between these exponents known as the \emph{hyperscaling relations} are expected to hold, namely
\begin{equation}
\label{eq:hyperscaling}
d \rho = \delta + 1 \qquad \text{ and } \qquad  d \nu = \beta(\delta+1).
\end{equation}
Note that the hyperscaling relations involve the dimension $d$ while the scaling relations do not. 
It is a heuristic originally due to Coniglio \cite{coniglio1985shapes} that the hyperscaling relations should hold if there are typically $O(1)$ `large' critical clusters on any given scale. This condition is believed to hold below the upper critical dimension but not above; see \cite{MR1716769,MR1431856} for detailed discussions. See \cite[Section 9.1]{grimmett2010percolation} for an overview of the heuristic arguments in support of the scaling and hyperscaling relations.

For nearest-neighbour percolation on two-dimensional planar lattices, the scaling relations \eqref{eq:scaling} and hyperscaling relations \eqref{eq:hyperscaling} were proven to hold by Kesten \cite{MR879034} under the assumption that the exponents $\delta$ and $\nu$ are both well-defined. (Kesten's results were of central importance to the subsequent computation of the critical exponents for site percolation on the triangular lattice following Smirnov's proof of conformal invariance \cite{smirnov2001critical2,smirnov2001critical,lawler2002one}.)   See also \cite{MR3940769} for related results on two-dimensional Voronoi percolation.   Meanwhile, in high dimensions, it is now known that the exponents $\beta,\gamma,\delta,\Delta,\eta,\rho,$ and $\nu$ all take their mean-field values in nearest-neighbour percolation with $d\gg 6$, from which it follows that the scaling relations \eqref{eq:scaling} are satisfied but that the hyperscaling relations \eqref{eq:hyperscaling} are violated;
see \cite{heydenreich2015progress} for an overview and \cite{MR1043524,MR1959796,MR2748397,MR3306002,MR762034,MR1127713} for highlights of the high-dimensional literature.

It remains completely open to prove that the scaling and hyperscaling relations hold in dimensions $2<d\leq 6$, even if one assumes that all the relevant exponents are well-defined. The most significant progress is due to Borgs, Chayes, Kesten, and Spencer \cite{MR1716769,MR1868996}, who proved in particular that the scaling and hyperscaling relations both hold in low-dimensional lattices for which $\rho$ is well-defined under the (as yet unproven) assumption that the number of clusters crossing the box $[0,r]\times [0,3r]^{d-1}$ in the easy direction is tight as $r\to\infty$. 
Their proof also yields that the hyperscaling \emph{inequalities}  
\[
d\rho \geq \delta+1 \qquad \text{ and } \qquad d-2+\eta \geq 2/\rho 
\]
hold on any graph for which these exponents are well-defined. Many further works have established various other inequalities between critical exponents; see the work of Tasaki \cite{MR918404,MR926203} for hyperscaling inequalities and the recent work \cite{1901.10363} and references therein for an overview of scaling inequalities. 

The main goal of this section is to prove the following theorem, which improves significantly upon the naive bound of \eqref{eq:primitive}.

\begin{thm}
\label{thm:hyperscalingsimple}
There exists a universal constant $C$ such that the following holds. 
 Let $G=(V,E,J)$ be a weighted graph, let $\beta \geq 0$, and suppose that there exist constants $A<\infty$ and $0 \leq \theta \leq 1/2$ such that $\bP_\beta( |K_u| \geq \lambda) \leq A \lambda^{-\theta}$ for every $u\in V$ and $\lambda>0$. Then
\[
\frac{1}{|\Lambda|}\sum_{v \in \Lambda} \bP_\beta(u\leftrightarrow v) 
\leq 
C A^{2 /(1+\theta)} |\Lambda|^{-2\theta/(1+\theta)}
\]
for every $u\in V$ and every finite set $\Lambda \subseteq V$.
\end{thm}

 In the context of $\Z^d$, it follows from \cref{thm:hyperscalingsimple} that if the exponents $\eta$ and $\delta$ are both well-defined then they satisfy the hyperscaling inequality
\begin{equation}
\label{eq:hyperscaling_eta_delta}
(2 -\eta)(\delta+1) \leq d(\delta-1).
\end{equation}
Indeed, if $\eta$ and $\delta$ are both well-defined then either $\eta \geq 2$, in which case \eqref{eq:hyperscaling_eta_delta} is trivial, or we can apply \cref{thm:hyperscalingsimple} with $\theta=1/\delta-\eps$ for $\eps>0$ arbitrary (noting that $\delta \geq 2$ when it is well defined \cite{aizenman1987sharpness,duminil2015new}) to compute that
\[
r^{-d+2-\eta} \approx r^{-d} \sum_{x \in \Lambda_r} \|x\|^{-d+2-\eta} \approx r^{-d} \sum_{x,y \in \Lambda_r} \bP_\beta(0\leftrightarrow x) \lesssim   r^{-2d/(\delta+1)} \qquad \text{ as $r\to \infty$,}
\]
 where we write $\lesssim$ to mean that the ratio of the logarithms of the left and right hand sides has limit supremum less than $1$. This inequality may be  rearranged to prove the inequality \eqref{eq:hyperscaling_eta_delta} in the case $\eta < 2$. We remark that the inequality \eqref{eq:hyperscaling_eta_delta}  is expected to be an equality below the upper critical dimension, as would follow from the validity of the scaling and hyperscaling relations.

In order to prove \cref{thm:hyperscalingsimple}, we first prove in \cref{subsec:universaltightness} a universal inequality implying in particular that the maximum size of the intersection of a cluster with a finite set is exponentially unlikely to be much larger than its median value, \cref{thm:universaltightness}. 
This inequality is proven by a combinatorial argument using the BK inequality. 
We then deduce \cref{thm:hyperscalingsimple}  from this inequality in \cref{subsec:hyperscalingproof} by a fairly straightforward calculation. 

\subsection{Universal tightness of the maximum cluster size in a finite region}
\label{subsec:universaltightness}

Let $G=(V,E,J)$ be a countable weighted graph, and consider Bernoulli bond percolation on $G$ with parameter $\beta \geq 0$. For each finite subset $\Lambda$ of $V$, we define 
\[
|K_\mathrm{max}(\Lambda)|=\max\{|K_v \cap \Lambda| : v\in V\}=\max\{|K_v \cap \Lambda| : v\in \Lambda\}.
\]
(This is a slight abuse of notation: there may be more than one cluster achieving this maximum, so that $K_\mathrm{max}(\Lambda)$ need not be well-defined as a set in general.) In this section we prove a general inequality, applying universally to all $G$, $\beta$, and $\Lambda$, implying that $|K_\mathrm{max}(\Lambda)|$ is of the same order as its `typical value' $M_\beta(\Lambda):=\min \{n \geq 0 : \bP_\beta(|K_\mathrm{max}(\Lambda)| \geq n)\leq e^{-1} \}$ with high probability. In particular, one simple consequence of this theorem is that $e^{-1} M_\beta(\Lambda) \leq \bE_\beta |K_\mathrm{max} (\Lambda)| \leq 10 M_\beta(\Lambda)$, so that the mean and typical value of $|K_\mathrm{max}(\Lambda)|$ are always of the same order. 
 We expect that the inequalities we prove in this section will have many further applications in the future.

\begin{theorem}[Universal tightness of the maximum cluster size] 
\label{thm:universaltightness}
Let $G=(V,E,J)$ be a countable weighted graph and let $\Lambda \subseteq V$ be finite and non-empty. Then the inequalities
\begin{align}
\bP_\beta\Bigl(|K_\mathrm{max}(\Lambda)| \geq \alpha M_\beta(\Lambda)\Bigr) &\leq \exp\left(-\frac{1}{9}\alpha \right)
\label{eq:BigClusterUnrooted}
\\
\text{and} \qquad \bP_\beta\Bigl(|K_\mathrm{max}(\Lambda)| < \eps M_\beta(\Lambda) \Bigr) &\leq 27 \eps \qquad \phantom{\text{and}} \qquad
\label{eq:SmallMaximum}
\end{align}
hold for every $\beta\geq 0$, $\alpha \geq 1$, and $0<\eps \leq 1$. Moreover, the inequality
\begin{equation}
\label{eq:BigClusterRooted}
\bP_\beta\Bigl(|K_u \cap \Lambda| \geq \alpha M_\beta(\Lambda)\Bigr) \leq e \cdot \bP_\beta\Bigl(|K_u \cap \Lambda| \geq  M_\beta(\Lambda)\Bigr) \exp\left(-\frac{1}{9}\alpha \right)
\end{equation}
holds for every $\beta \geq 0$, $\alpha \geq 1$, and $u \in V$.
\end{theorem}

We will deduce this theorem as a corollary of the following more general inequality.

\begin{theorem}
\label{thm:maximum_tail}
Let $G=(V,E,J)$ be a countable weighted graph and let $\Lambda \subseteq V$ be finite and non-empty. Then the inequalities
\begin{align}
&&\bP_\beta\bigl(|K_\mathrm{max}(\Lambda)| \geq 3^k \lambda \bigr) &\leq \bP_\beta\bigl(|K_\mathrm{max}(\Lambda)| \geq \lambda  \bigr)^{3^{k-1}+1}
\label{eq:maximum_tail_unrooted}
 \\
\text{and} &&\bP_\beta\bigl(|K_u \cap \Lambda| \geq 3^k \lambda \bigr) &\leq \bP_\beta\bigl(|K_\mathrm{max}(\Lambda)| \geq \lambda \bigr)^{3^{k-1}} \bP_\beta\bigl(|K_u \cap \Lambda| \geq \lambda \bigr)
\label{eq:maximum_tail_rooted}
\end{align}
hold for every $\beta \geq 0$, $\lambda \geq 1$, $k\geq 0$, and $u\in V$.
\end{theorem}

(This theorem does \emph{not} require $\lambda$ to be an integer.)

\begin{proof}[Proof of \cref{thm:universaltightness} given \cref{thm:maximum_tail}]
The inequalities \eqref{eq:BigClusterUnrooted} and \eqref{eq:BigClusterRooted} are trivial when $\alpha \leq 9$, while for $\alpha \geq 9$ they follow immediately from \eqref{eq:maximum_tail_unrooted} and \eqref{eq:maximum_tail_rooted} by taking $\lambda=M_\beta(\Lambda)$ and $k=\lfloor \log_3 \alpha \rfloor \geq 1$ and using that $3^{\lfloor \log_3 \alpha \rfloor-1} \geq \alpha /9$. We now turn to \eqref{eq:SmallMaximum}. Write $M=M_\beta(\Lambda)$. The inequality is trivial if $\eps M <1$ or $9 \eps \geq 1$, so we may assume that $M \geq 1/ \eps \geq 9$. The definitions ensure that $\bP_\beta(|K_\mathrm{max}(\Lambda)|\geq M-1) \geq e^{-1}$. Let $k=\lfloor \log_3 (1/\eps)\rfloor-1$, so that  $3^{-k}(M-1) \geq 3\eps (M-1) \geq \eps M$. The inequality \eqref{eq:maximum_tail_unrooted}
implies that
\[
\bP_\beta\Bigl(|K_\mathrm{max}(\Lambda)| \geq M -1 \Bigr)
\leq 
\bP_\beta\Bigl(|K_\mathrm{max}(\Lambda)| \geq 3^{-k} (M-1) \Bigr)^{3^{k-1}} \leq 
\bP_\beta\Bigl(|K_\mathrm{max}(\Lambda)| \geq  \eps M \Bigr)^{1/(27\eps)},
\]
which can be rearranged to yield that
\[
\bP_\beta\Bigl(|K_\mathrm{max}(\Lambda)| <  \eps M \Bigr) \leq 1-\bP_\beta\Bigl(|K_\mathrm{max}(\Lambda)| \geq M -1 \Bigr)^{27\eps} \leq 1- e^{-27\eps} \leq 27\eps
\]
as claimed, where we used that $1-e^{-x} \leq x$ in the third inequality.
\end{proof}

We now turn to the proof of \cref{thm:maximum_tail}. We will deduce the theorem as a consequence of the BK inequality together with the following combinatorial lemma.

\begin{lemma}
\label{lem:divide_and_conquer}
Let $G=(V,E)$ be a connected, locally finite graph, let $k\geq 1$, and let $A$ be a finite subset of $V$ such that $|A| \geq 3^k$. Then there exists $m \geq 3^{k-1}+1$ and a collection $\{E_i : 1 \leq i \leq m\}$ of disjoint, non-empty subsets of $E$ such that the following hold:
\begin{enumerate}
  \item For each $1\leq i \leq m$, the subgraph of $G$ spanned by $E_i$ is connected.
  \item Every vertex in $V$ is incident to some edge in $\bigcup_{i=1}^{m} E_i$.
  \item The set $V_i$ of vertices incident to an edge of $E_i$ satisfies
\[
3^{-k}  \leq \frac{|A \cap V_i|}{|A|} < 3^{-k+1}
\]
for each $1 \leq i \leq m$.
\end{enumerate}
\end{lemma}

When $G$ is finite, the proof of this lemma can be used to derive an explicit divide-and-conquer algorithm for finding such a collection of sets $E_1,\ldots,E_{m}$ after taking a spanning tree of $G$.

\begin{proof}[Proof of \cref{lem:divide_and_conquer}]
We may without loss of generality assume that $G=T$ is a tree, taking a spanning tree of $G$ otherwise.  In this case we will prove the stronger claim that the sets $\{E_i : 1 \leq i \leq m\}$ can be taken to be a partition of $E$. (Here, a \textbf{partition} of $E$ is a set of disjoint subsets of $E$ whose union is $E$.)
 We say that a partition $\pi$ of $E$ is \textbf{good} if each piece of $\pi$ spans a connected subgraph of $T$.


We first prove that if $T=(V,E)$ is a locally finite tree and $A \subseteq V$ has $3\leq |A| < \infty$ then there exists a good partition of $E$ into two non-empty sets $E_1$ and $E_2$ such that if $V(E_i)$ denotes the set of vertices incident to at least one edge of $E_i$ then
\[
\frac{1}{3}|A| \leq |A \cap V(E_i)| \leq \left\lceil \frac{2}{3}|A|\right\rceil
\]
for each $i=1,2$.
 Let $\rho$ be a vertex of $T$. We root $T$ at $\rho$, and call a vertex $v$ a \textbf{descendant} of an edge $e$ if the unique shortest path from $\rho$ to $v$ contains $e$. We will iteratively define a sequence $(v_n,W_n)_{n= 0}^N$, where $1 \leq N \leq \infty$, $v_n \in V$, and $W_n \subseteq E$ for each $0 \leq n \leq N$.  Start by setting $v_0=\rho$ and $W_0=\emptyset$. At each intermediate stage $0<n<N$ of the sequence, $W_n$ and $W_n^c$ will both be non-empty and span connected subgraphs of $T$, while $v_n$ will be incident to edges of both $W_n$ and $W_n^c$. Given $(v_n,W_n)$ for some $n\geq 0$, we let $V_n$ be the set of vertices of $T$ that are either equal to $\rho$ or incident to some edge of $W_n$, and define $(v_{n+1},W_{n+1})$ using the following procedure, which is illustrated in \cref{fig:centroid_example}.
 \begin{figure}[t]
 \centering
\includegraphics[width=\textwidth]{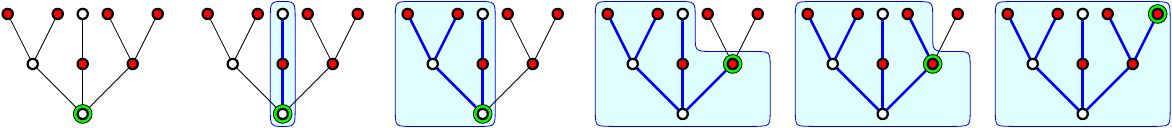}
\caption{A sequence of good partitions of a tree as constructed in the proof of \cref{lem:divide_and_conquer}. Red vertices represent elements of the distinguished set $A$. At each stage, edges belonging to $W_n$ are thick, blue, and contained in the blue shaded region, while edges belonging to the complement $W_n^c$ are thin and black. The vertex $v_n$, which lies at the boundary of $W_n$ and $W_n^c$, is represented with a green outer ring. In this example, $W_n$ is first incident to more than $1/3$ of the vertices of $A$ when $n=2$, in which case $W_n$ and $W_n^c$ are both incident to exactly $3$ vertices of $A$.}
\label{fig:centroid_example}
\end{figure}
\begin{enumerate}
\item If $v_n$ has exactly one edge $e\in W_n^c$ adjacent to it, we set $W_{n+1} = W_n \cup \{e\}$ and set $v_{n+1}$ to be the other endpoint of this edge. If $W_{n+1}=E$ then we set $N=n+1$ and terminate the sequence.
\item Otherwise, $v_{n+1}$ has at least two edges of $W_n^c$ adjacent to it. Enumerate these edges $e_1,\ldots,e_\ell$, and let $D_i$ be the set of descendants of $e_i$ for each $1\leq i \leq \ell$. Since $\sum_{i=1}^\ell |D_i \cap A| = |V_n^c \cap A|$, there must exist $1\leq i \leq \ell$ such that $|D_i \cap A| \leq |V_n^c \cap A|/2$. Choose one such $i$, set $W_{n+1}$ to be the union of $W_n$ with the set of edges incident to $D_i$ (i.e., having at least one endpoint in $D_i$), and set $v_{n+1}=v_n$.
\end{enumerate}

We may verify by induction that $W_n$ and $W_n^c$ are indeed both non-empty and span connected subgraphs of $T$ for every $0< n < N$ as claimed, so that $\{W_n,W_n^c\}$ is a good partition of $E$ for every $0< n < N$. Moreover, the assumption that $T$ is locally finite implies that $\bigcup_{n=0}^N W_n = E$ and hence that $\bigcup_{n=0}^N V_n = V$. Since $A$ is finite and $\bigcup_{n=0}^N V_n = V$, there exists a finite time $N'\leq N$ such that $V_n$ contains $A$ for every $ N' \leq n \leq N$. 
Observe that the set $\{0\leq  n \leq N' : |V_n \cap A| > |A|/3\}$ contains $N'$ but does not contain $0$ since $|A| \geq 3$. Letting $m \geq 1$ be the minimal element of this set, we have that
\begin{multline*}
\frac{1}{3}|A| < |V_{m} \cap A| \leq |V_{m-1} \cap A| + \max \Bigl\{ 1,\, \frac{1}{2}|V_{m-1}^c \cap A| \Bigr\} \\= \max \Bigl\{ |V_{m-1} \cap A| + 1,\, \frac{1}{2}\left(|A|+|V_{m-1} \cap A|\right) \Bigr\} \leq \frac{2}{3}|A|,
\end{multline*}
where we used that $|A| \geq 3$ in the final inequality. It follows in particular that $0 < m < N$.
Moreover, if $V_m'$ denotes the set of vertices incident to an edge of $W_m^c$ then $V_m \cup V_m'=V$ and $|V_m \cap V_m'|\leq 1$, so that
\[
\frac{1}{3} |A| \leq |A| - |V_m \cap A| \leq |V_m'\cap A| \leq |A|- |V_m \cap A| +1 \leq \left\lceil \frac{2}{3}|A|\right\rceil,
\]
where the final inequality follows since $|V_m \cap A| > \frac{1}{3}|A|$ is an integer. It follows that $\{W_{m},W_{m}^c\}$ is a good partition of $E$ with the desired properties.


We now apply the claim proven in the previous paragraph to complete the proof of the lemma. 
Let $T=(V,E)$ be a locally finite tree, let $k\geq 1$, and let $A \subseteq V$ satisfy $3^k \leq |A| < \infty$. For each subset $K$ of $E$, write $V(K)$ for the set of vertices incident to an edge of $K$. Construct a sequence $\pi_0,\pi_1,\ldots$ of good partitions of $E$ recursively as follows. Set $\pi_0=\{E\}$ to be the trivial partition. For each $i\geq 0$, define $\pi_{i+1}$ by retaining those pieces $W$ of $\pi_i$ satisfying $|V(W) \cap A| < 3^{k-1} |A|$ and splitting each piece $W$  of $\pi_i$ with $|V(W) \cap A| \geq 3^{-k+1}|A|$ into two pieces $W_1$ and $W_2$ that each span connected subgraphs of $T$ and satisfy 
\[ \frac{1}{3}|V(W) \cap A| \leq |V(W_1) \cap A|,|V(W_2) \cap A| \leq \left \lceil\frac{2}{3}|V(W) \cap A|\right\rceil.\] This can be done by applying the claim proven in the previous paragraph to the subgraph of $T$ spanned by the piece $W$. We have by induction on $i$ that 
\begin{align*}
\min \{V(W) \cap A : W \in \pi_{i} \} \geq 3^{-k} |A| 
\end{align*}
for every $i\geq 0$. Moreover, noting that $\lceil 2x/3 \rceil < 9x/10$ for every integer $x\geq 3$, we also have by induction on $i$ that
\begin{align*}
\max \{V(W) \cap A : W \in \pi_{i} \} < \max\left\{3^{-k+1}|A|,\, \left(\frac{9}{10}\right)^i |A| \right\}
\end{align*}
for every $i \geq 0$. It follows that there exists $i_0<\infty$ such that every piece $W$ of the good partition $\pi_{i_0}$ satisfies $3^{-k}|A| \leq |V(W) \cap A| < 3^{-k+1}|A|$ as required. 
Letting $m=|\pi_{i_0}|$ we have that $3^{-k+1} |A| m > \sum_{W \in \pi_{i_0}} |V(W) \cap A| \geq |A|$ and hence that $m \geq 3^{k-1}+1$ as desired.
\end{proof}


\begin{proof}[Proof of \cref{thm:maximum_tail}]
Let $G=(V,E,J)$ be a finite weighted graph. 
 Recall that
 if $A_1,\ldots,A_k$ are (not necessarily distinct), increasing subsets of $\{0,1\}^E$, the \textbf{disjoint occurrence} $A_1 \circ \cdots \circ A_k$ is the set of $\omega \in \{0,1\}^E$ such that there exist disjoint sets $W_1,\ldots,W_k \subseteq \{e:\omega(e)=1\}$ such that
 \[
(\omega'(e)=1 \text{ for every $e\in W_i$}) \Rightarrow (\omega' \in A_i) \qquad  \text{for every $\omega' \in \{0,1\}^E$ and $1 \leq i \leq k$.}
 \]
 (Here, a subset of $\{0,1\}^E$ is said to be \textbf{increasing} if $\omega \in A \Rightarrow \omega' \in A$ for every $\omega,\omega' \in \{0,1\}^E$ such that $\omega'(e)\geq \omega(e)$ for every $e\in E$.)
The sets $W_1,\ldots,W_k$ are known as \textbf{disjoint witnesses} for the events $A_1,\ldots,A_k$. The van den Berg and Kesten inequality \cite{MR799280}, a.k.a.\ the \textbf{BK inequality}, states that if $G=(V,E,J)$ is a finite weighted graph and $A_1,\ldots,A_k \subseteq \{0,1\}^E$ are increasing events then
\[
\bP_\beta(A_1 \circ \cdots \circ A_k) \leq \prod_{i=1}^k \bP_\beta(A_i)
\]
for every $\beta \geq 0$. See \cite[Chapter 2.3]{grimmett2010percolation} for further background.

Let $G=(V,E,J)$ be a finite weighted graph and let $\Lambda \subseteq V$. Suppose that the event $\{|K_\mathrm{max}(\Lambda)|\geq 3^k \lambda\}$ holds for some $\lambda \geq 1$ and $k\geq 1$, and let $v\in V$ be such that $|K_v \cap \Lambda| \geq 3^k \lambda$. Applying \cref{lem:divide_and_conquer} to $K_v$ yields that there exists $m\geq 3^{k-1}+1$ and $m$ disjoint sets of open edges $E_1,\ldots,E_{m}$, each spanning a connected subgraph of $K_v$, such that the set $V_i$ of vertices incident to an edge of $E_i$ satisfies $|V_i \cap \Lambda| \geq \lambda$ for every $1\leq i \leq m$. It follows that the sets $E_1,\ldots,E_{m}$ are all witnesses for the event $\{|K_\mathrm{max}(\Lambda)|\geq \lambda\}$, and since these sets are all disjoint we deduce that
\begin{equation}
\label{eq:BK_inclusion1}
\{
|K_\mathrm{max}(\Lambda)|\geq 3^k \lambda
\} \subseteq \underbrace{\{|K_\mathrm{max}(\Lambda)|\geq \lambda\} \circ \cdots \circ \{|K_\mathrm{max}(\Lambda)|\geq \lambda\}}_{\text{$3^{k-1}+1$ copies}}
\end{equation}
for every $\lambda \geq 1$ and $k\geq 1$.  Taking probabilities on both sides and applying the BK inequality yields the claimed inequality \eqref{eq:maximum_tail_unrooted} in the case that $G$ is finite. Now suppose that the event $\{|K_u \cap \Lambda|\geq 3^k \lambda\}$ holds for some $\lambda \geq 1$, $k\geq 1$, and $u\in V$. Similarly to above, applying \cref{lem:divide_and_conquer} to $K_u$ yields that there exists $m \geq 3^{k-1}+1$ and $m$ disjoint sets of open edges $E_1,\ldots,E_{m}$, each spanning a connected subgraph of $K_u$, such that the set $V_i$ of vertices incident to an edge of $E_i$ satisfies $|V_i \cap \Lambda| \geq \lambda$ for every $1\leq i \leq m$ and such that $\bigcup_{i=1}^{m} V_i$ is equal to the vertex set of $K_u$. In particular, $u\in V_i$ for some $1 \leq i \leq m$. Thus, at least one of the sets $E_1,\ldots,E_{m}$ is a witness for the event $\{
|K_u \cap \Lambda|\geq \lambda
\}$, while the remaining sets are all witnesses for the event $\{|K_\mathrm{max}(\Lambda)|\geq \lambda\}$. Since the sets $E_1,\ldots,E_{m}$ are all disjoint, we deduce that
\begin{equation}
\label{eq:BK_inclusion2}
\{
|K_u \cap \Lambda|\geq 3^k \lambda
\} \subseteq \{
|K_u \cap \Lambda|\geq \lambda
\} \circ \underbrace{\{|K_\mathrm{max}(\Lambda)|\geq \lambda\} \circ \cdots \circ \{|K_\mathrm{max}(\Lambda)|\geq \lambda\}}_{\text{$3^{k-1}$ copies}}
\end{equation}
for every $\lambda \geq 1$, $k\geq 1$, and $u \in V$. As before, taking probabilities on both sides  and applying the BK inequality yields the claimed inequality \eqref{eq:maximum_tail_rooted} in the case that $G$ is finite. The infinite cases of \eqref{eq:maximum_tail_unrooted} and \eqref{eq:maximum_tail_rooted} follow straightforwardly from the finite cases by passing to the limit in an exhaustion over finite subgraphs.
\end{proof}


\subsection{Proof of the hyperscaling inequality}
\label{subsec:hyperscalingproof}

We now apply \cref{thm:universaltightness} to prove \cref{thm:hyperscalingsimple}. In fact we will prove the following stronger theorem which also gives control of the maximal cluster size in $\Lambda$ and allows $0\leq \theta<1$.

\begin{thm}
\label{thm:hyperscaling}
There exists a universal continuous function $C:[0,1)\to (0,\infty)$ such that the following holds.
 Let $G=(V,E,J)$ be a countable weighted graph, let $\beta\geq 0$, let $\Lambda \subseteq V$ be finite, and suppose that there exist $A<\infty$ and $0\leq \theta <1$ such that $\bP_\beta( |K_u \cap \Lambda| \geq n) \leq A n^{-\theta}$ for every $u\in V$ and $n \geq 1$. Then 
\[
M_\beta(\Lambda) \leq C(\theta) A^{1/(1+\theta)} |\Lambda|^{1/(1+\theta)} \quad \text{ and } \quad 
\frac{1}{|\Lambda|}\sum_{v \in \Lambda} \bP_\beta(u\leftrightarrow v) 
\leq 
C(\theta) A^{2/(1+\theta)} |\Lambda|^{-2\theta/(1+\theta)}
\]
for every $u\in V$.
\end{thm}

We begin by writing down the following immediate corollary of \cref{thm:universaltightness}.


\begin{corollary}
\label{cor:obvious}
 Let $G=(V,E,J)$ be a weighted graph and let $\beta\geq 0$. Let $u\in V$ and $\Lambda \subseteq V$ be finite, and suppose that there exist constants $A<\infty$ and $0\leq \theta <1$ such that $\bP_\beta( |K_u \cap \Lambda| \geq n) \leq A n^{-\theta}$ for every $n \geq 1$. Then 
 \vspace{-0.3em}
 \[
\bP_\beta(|K_u \cap \Lambda| \geq n) \leq e A \left(\frac{18}{n}\right)^{\theta} \!\exp\left[-\frac{n}{18 M_\beta(\Lambda)}\right] \qquad \text{for every $n\geq 1$.}
 \]
\end{corollary}

\begin{proof}[Proof of \cref{cor:obvious}]
Write $M=M_\beta(\Lambda)$. 
The claim is trivial when $n \leq 18 M$. If not, we have by \cref{thm:universaltightness} that
\begin{equation*}
\bP(|K_u \cap \Lambda| \geq n) \leq e A M^{-\theta} \exp\left[-\frac{n}{9 M}\right] \leq e A n^{-\theta} \exp\left[-\frac{n}{18 M}\right] \left(\frac{n^{\theta}}{M^{\theta}}\exp\left[-\frac{n}{18M}\right] \right).
\end{equation*}
Using that $x^{\theta} e^{-x/C}$ is decreasing on $[C,\infty)$ yields the claimed inequality.
\end{proof}

\begin{proof}[Proof of \cref{thm:hyperscaling}]
 For each $u\in V$ we can apply \cref{cor:obvious} to compute that
 \begin{align}
 \nonumber
\sum_{v \in \Lambda} \bP_\beta(u\leftrightarrow v) &= \bE_\beta |K_u \cap \Lambda| = \sum_{n\geq 1} \bP_\beta(|K_u \cap \Lambda| \geq n) \leq e A \sum_{n=1}^\infty \left(\frac{18}{n}\right)^{\theta}  \exp\left[-\frac{n}{18 M}\right] \\ &\leq e A \int_0^\infty \left(\frac{18}{t}\right)^{\theta}  \exp\left[-\frac{t}{18 M}\right] \dif t 
= 18 e  \Gamma(1-\theta) A M^{1-\theta} 
\label{eq:lambda_susceptibility}
 \end{align}
 where $\Gamma(\alpha)=\int_0^\infty t^{\alpha-1} e^{-t} \dif t$ is the Gamma function and where we used the change of variables $s=t/(18M)$ in the final equality.
 Summing over $u\in \Lambda$, it follows that
\begin{equation}
\label{eq:max_to_expectation}
\sum_{u,v \in \Lambda} \bP_\beta(u \leftrightarrow v) = \sum_{u\in \Lambda} \bE_\beta |K_u \cap \Lambda| \leq 18 e  \Gamma(1-\theta) A M^{1-\theta}|\Lambda|.
\end{equation}
On the other hand, we also have the lower bound
\begin{align}
\sum_{u,v \in \Lambda} \bP_\beta(u \leftrightarrow v) &= \E \left[\sum_{u,v \in \Lambda} \mathbbm{1}(u \leftrightarrow v)\right] \geq 
\bE_\beta \left[|K_\mathrm{max}(\Lambda)|^2\right]
\nonumber
\\ &\geq 
 (M-1)^2 \bP_\beta \Bigl(|K_\mathrm{max}(\Lambda)|\geq M-1\Bigr) \geq \frac{1}{e}(M-1)^2 \geq \frac{1}{4e}M^2,
 \label{eq:Max2ndmoment}
\end{align}
where we used that $M\geq 2$ in the final inequality.
Comparing the estimates \eqref{eq:max_to_expectation} and \eqref{eq:Max2ndmoment} and rearranging yields that
\begin{equation}
\label{eq:MaxBound}
M^{1+\theta} \leq 72 e^2 \Gamma(1-\theta) A |\Lambda|,
\end{equation}
completing the proof of the first claimed bound. Substituting this bound into \eqref{eq:lambda_susceptibility} yields that there exists a universal continuous function $C:[0,1)\to(0,\infty)$ such that
\begin{equation*}
\frac{1}{|\Lambda|}\sum_{v \in \Lambda} \bP_\beta(u \leftrightarrow v)  \leq \frac{1}{|\Lambda|} 18 e  \Gamma(1-\theta) A
\left(72 e^2 \Gamma(1-\theta) A |\Lambda|\right)^{(1-\theta)/(1+\theta)} \\ = C(\theta) A^{2/(1+\theta)}|\Lambda|^{-2\theta/(1+\theta)}
\end{equation*}
for each $u\in V$, completing the proof of the second bound.
\end{proof}


\begin{remark}
Although the distribution of the entire cluster of critical percolation on a transitive weighted graph always satisfies $\bP_{\beta_c}(|K_v|\geq n) \geq c n^{-1/2}$, the $1/2<\theta <1$ case of \cref{thm:hyperscaling} may nevertheless be useful when  taking e.g.\ $\Lambda \subseteq \Z^{d-k} \subseteq \Z^d$ to be contained in a lower-dimensional subspace of the full lattice.
In particular, it would be interesting if one could improve the high-dimensional case of \cref{thm:main} by first proving an upper bound of the form $\bP_{\beta_c}(|K_0 \cap \Z^{d-2}| \geq n) \leq n^{-1/\delta_2}$ for some $\delta_2 < 2$ and then using \cref{thm:hyperscaling} to get an improved bound on the two-point function within $\Z^{d-2}$. It seems that only a relatively modest improvement along these lines is needed to give a lace expansion-free proof that the triangle condition is satisfied when $d$ is large and $\alpha$ is fixed. 
 Note also that bounds on the maximum cluster size similar to those of \cref{thm:hyperscaling} may be proven in the regime $\theta \geq 1$ by following the proof as above but considering $\sum_{u\in \Lambda} \bE_\beta |K_v \cap \Lambda|^k$ instead of $\sum_{u\in \Lambda} \bE_\beta |K_v \cap \Lambda|$ for appropriate choice of $k\geq 2$.
\end{remark}







\section{An improved two-ghost inequality}
\label{sec:twoghost}

In this section we derive an improved version of the two-ghost inequality of \cite[Theorem 1.6 and Corollary 1.7]{1808.08940} as stated for long-range models in \cite[Section 3]{Hutchcroft2020Ising}. This improved two-ghost inequality will be applied together with \cref{thm:hyperscalingsimple} to prove \cref{thm:main,thm:general} in the next section. The proof of the two-ghost inequality uses ideas originating in the important work of Aizenman, Kesten, and Newman \cite{MR901151}; see \cite{1808.08940} and \cite{MR3395466} for further discussion of how the methods of \cite{MR901151} can be used to derive quantitative estimates on critical percolation. Our improvement to the two-ghost inequality as stated in \cite[Corollary 3.2]{Hutchcroft2020Ising} is two-fold: 
\begin{itemize}
\item We show that a starting assumption of the form $\bP_\beta(|K|\geq n)\leq A n^{-\theta}$ (as will come from our bootstrapping hypothesis) can be used to improve the exponent given by the two-ghost inequality. The fact that this can be done had previously been discussed briefly \cite[Remark 6.1]{1808.08940} and \cite[Remark 3.6]{Hutchcroft2020Ising}.
\item We use a re-weighting trick to improve the bound one obtains on the probability of the two-arm event for a typical `long' edge. The basic idea behind this improvement is that the two-ghost inequality of \cite{Hutchcroft2020Ising} holds not just for the weights $J$ that are given with the graph $G$, but also for any other automorphism-invariant choice of weights. Optimizing the resulting bound over all possible automorphism-invariant weights  leads to the bound of \cref{thm:two_ghost_S}.
\end{itemize}

For the benefit of future applications, we phrase the results in this section not just for Bernoulli percolation but for the more general class of \emph{percolation in random environment} models. The same level of generality was employed in \cite[Section 3]{Hutchcroft2020Ising}, where we applied the two-ghost inequality to the random-cluster and Ising models. (See in particular \cite[Section 3.3]{Hutchcroft2020Ising} for a representation of the random-cluster model as a percolation in random environment model first arising in \cite{bollobas1996random}.) 
Let $G=(V,E,J)$ be a countable weighted graph. 
Suppose that $\mu$ is a probability measure on $[0,1]^E$, and let $\bp=(\bp_e)_{e\in E}$ be a $[0,1]^E$-valued random variable with law $\mu$. Let $(U_e)_{e\in E}$ be i.i.d.\ Uniform$[0,1]$ random variables independent of $\mathbf{p}$ and let $\omega=\omega(\bp,U)$ be the $\{0,1\}^E$-valued random variable defined by
$\omega(e) = \mathbbm{1}(U_e \leq \bp_e)$ for each $e\in E$.
 We say that $\omega$ is a \textbf{percolation in random environment} on $G$ with environment distribution $\mu$ and write $\bP_\mu$ for the joint law of $\bp$ and $\omega$. 
We can consider Bernoulli percolation on $G$ to be a percolation in random environment model for which the environment measure $\mu$ is concentrated on the point $(\bp_e)_{e\in E} = (1-e^{-\beta J_e})_{e\in E}$.

For each $e\in E$ and $n\geq 1$, let $\sS'_{e,n}$ be the event that the endpoints of $e$ belong to distinct clusters each of which include at least $n$ vertices and at least one of which is finite\footnote{Note while similar notation appeared in \cref{subsec:short_proof}, we are now using slightly different notation in which we index by edges rather than vertices. This is more natural in the more general context we are working in here.}. (We use $\sS'_{e,n}$ rather than $\sS_{e,n}$ to indicate that we are measuring volume in terms of vertices rather than edges.) Recall that we write $E_v^\rightarrow$ for the set of oriented edges emanating from $v$ for each vertex $v$ of $G$; we do not distinguish notationally between oriented and unoriented edges, and will often abuse notation to apply functions defined on unoriented edges to oriented edges by forgetting the orientation.

 \begin{theorem}[Improved two-ghost inequality]
\label{thm:two_ghost_S} Let $G=(V,E,J)$ be a connected transitive weighted graph, let $o$ be a vertex of $G$, and let $\Gamma \subseteq \Aut(G)$ be a closed, transitive, unimodular subgroup of automorphisms of $G$. Let $\mu$ be a $\Gamma$-invariant probability measure on $[0,1]^E$
and suppose that there exist constants $A<\infty$ and $0\leq \theta <1/2$ such that $\bP_{\mu}(|K_o| \geq n) \leq A n^{-\theta}$ for every $n \geq 1$. Then
\begin{align}
\label{eq:improved_two_ghost}
\sum_{e\in E^\rightarrow_o} \bE_{\mu}\left[\mathbbm{1}(\sS_{e,n}') \sqrt{\frac{\bp_e}{1-\bp_e}}\right]^2 \leq \frac{40000 \cdot A^2}{(1-2\theta)^2 n^{1+2\theta}}  \qquad \text{ for every $n\geq 1$.}
\end{align}
\end{theorem}

\medskip

Here, the closed, transitive subgroup $\Gamma \subseteq \Aut(G)$ is said to be \textbf{unimodular} if it satisfies the \textbf{mass-transport principle}, i.e., if
\begin{equation}
\label{eq:MTP}
\sum_{v \in V} F(o,v) = \sum_{v\in V}F(v,o)
\end{equation}
for every $o\in V$ and every function $F:V^2\to [0,\infty]$ that is diagonally invariant under $\Gamma$ in the sense that $F(\gamma u,\gamma v)=F(u,v)$ for every $u,v\in V$ and $\gamma \in \Gamma$. Equivalently, $\Gamma$ is unimodular if its left and right Haar measures coincide. This holds in particular whenever $\Gamma$ is countable, in which case it has counting measure as both a left and right Haar measure.
Most transitive weighted graphs arising in examples have unimodular automorphism groups, including all amenable transitive weighted graphs and all weighted graphs defined in terms of a countable group as described after the statement of \cref{thm:general}. See e.g.\ \cite[Section 2]{Hutchcroft2020Ising} and \cite[Chapter 8]{LP:book} for further background and for proofs of these statements.
 For the main purposes of this paper, it suffices to consider the case that $G$ has vertex set $\Z^d$ and that $\Gamma=\Z^d$ acts transitively on $G$ by translations as in \cref{thm:main}.

Let $G=(V,E,J)$ be a connected, transitive weighted graph, let $o$ be a vertex of $G$, and let $\Gamma$ be a closed transitive subgroup of $\Aut(G)$.   We call $\mathrm{w}:E\to [0,1]$ a ($\Gamma$-)\textbf{good weight function} if  $\mathrm{w}(\gamma e)=\mathrm{w}(e)$ for every $e\in E$ and $\gamma \in \Gamma$, $\sum_{E^\rightarrow_o} \mathrm{w}(e)=1$, and $\sum_{E^\rightarrow_o} \sqrt{\mathrm{w}(e)} < \infty$. (The last condition holds trivially if $\mathrm{w}(e)=0$ for all but finitely many $e\in E^\rightarrow_o$, and it would in fact suffice to consider this case for the rest of the proof.)
 Let $\mu$ be a $\Gamma$-invariant probability measure on $[0,1]^E$, let $\bp$ be a random variable with law $\mu$ and let $\omega$ be the associated percolation in random environment process as above. 
Let $h>0$. Given the environment $\bp$ and a good weight function $\mathrm{w}$, let $\cG \in \{0,1\}^E$  be a random subset of $E$, independent of $\bp$ and $\omega$, where each edge $e \in E$ is included in $\cG$ independently at random with probability $1-e^{-h \mathrm{w}(e)}$
 of being included.   We write $\bP_{\mu,\mathrm{w},h}$ and $\bE_{\mu,\mathrm{w},h}$ for probabilities and expectations taken with respect to the joint law of $\bp$, $\omega$, and $\cG$. 
We call $\cG$ the $\mathrm{w}$-\textbf{ghost field} and call an edge $\mathrm{w}$-\textbf{green} if it is included in $\cG$. Note that
\[
\bP_{\mu,\mathrm{w},h}(A \cap \cG = \emptyset \mid \bp) = \exp\left[-h \cdot \mathrm{w}(A) \right] 
\]
for every finite set $A \subseteq E$, where we write $\mathrm{w}(A)=\sum_{e\in A}\mathrm{w}(e)$ for the total weight of $A$.


For each edge $e$ of $G$, we define $\sT_e$ to be the event that $e$ is closed in $\omega$ and that the endpoints of $e$ are in distinct clusters of $\omega$, each of which touches some $\mathrm{w}$-green edge, and at least one of which is finite. We will deduce \cref{thm:two_ghost_S} from the following proposition.

\begin{prop}
\label{prop:two_ghost_weights} Let $G=(V,E,J)$ be a connected transitive weighted graph,  let $o$ be a vertex of $G$, and let $\Gamma \subseteq \Aut(G)$ be a closed transitive unimodular subgroup of automorphisms of $G$. Let $\mu$ be a $\Gamma$-invariant probability measure on $[0,1]^E$
and suppose that there exist constants $A<\infty$ and $0\leq \theta <1/2$ such that $\bP_{\mu}(|K_o|\geq n) \leq A n^{-\theta}$ for every $n\geq 1$. Then for each $\Gamma$-good weight function $\mathrm{w}:E\to [0,1]$ we have that
\begin{align}
\label{eq:two_ghost}
\sum_{e\in E^\rightarrow_o} \sqrt{\mathrm{w}(e)}\bE_{\mu,\mathrm{w},h}\left[\mathbbm{1}(\sT_e) \sqrt{\frac{\bp_e}{1-\bp_e}}\right] \leq \frac{40 A}{1-2\theta} h^{(1+2\theta)/2} \qquad \text{for every $h>0$.}
\end{align}
\end{prop}

(The condition $\sum_{e\in E^\rightarrow_o} \sqrt{\mathrm{w}(e)} < \infty$ is not really needed for this proposition to hold, but will slightly simplify the proof.)
Before proving this theorem, let us see how it implies \cref{thm:two_ghost_S}.

\begin{proof}[Proof of \cref{thm:two_ghost_S} given \cref{prop:two_ghost_weights}]
Let $\mathrm{w}:E\to[0,1]$ be a $\Gamma$-good weight function.
Let $e$ be an edge of $G$ with endpoints $x$ and $y$ and let $\sD_e$ be the event that $x$ and $y$ are in distinct clusters at least one of which is finite. 
Then we have by the definitions that
\begin{multline*}
\mathbf{P}_{\mu,\mathrm{w},h}(\sT_e \mid \bp) 
\geq (1-e^{-\frac{1}{2}hn})^2\bP_{\mu,\mathrm{w},h}\bigl(\sD_e \cap \bigl\{\mathrm{w}(E(K_x)),\mathrm{w}(E(K_{y})) \geq \tfrac{1}{2}n\bigr\}\mid \bp \,\bigr) \\\geq   (1-e^{-\frac{1}{2}hn})^2\bP_{\mu,\mathrm{w},h}\bigl(\sS_{e,n}'\mid \bp \,\bigr)
\end{multline*}
for each $h>0$ and $n\geq 1$, 
where we used that $|A| \leq 2\mathrm{w}(E(A))$ for every $A \subseteq V$ in the final inequality. 
  Setting $h=c n^{-1}$ with $c\geq 1$ and applying \cref{prop:two_ghost_weights}, it follows that
\begin{align*}
\sum_{e\in E^\rightarrow_o} \sqrt{\mathrm{w}(e)} \bE_{\mu}\left[\mathbbm{1}(\sS_{e,n}') \sqrt{\frac{\bp_e}{1-\bp_e}}\right] &\leq (1-e^{-\frac{1}{2}hn})^{-2} \sum_{e\in E^\rightarrow_o} \sqrt{\mathrm{w}(e)} \bE_{\mu,\mathrm{w},h}\left[\mathbbm{1}(\sT_{e}) \sqrt{\frac{\bp_e}{1-\bp_e}}\right]\\
  &\leq   \frac{c}{(1-e^{-c/2})^2} \cdot \frac{40 A}{1-2\theta}   n^{-(1+2\theta)/2}
\end{align*}
for every $n\geq 1$ and $c\geq 1$.
Using that $\inf_{c \geq 1} 40 c (1-e^{-c/2})^{-2} =196.433 \ldots \leq 200$ gives that
\begin{align}
\label{eq:weighted_twoghost_final}
\sum_{e\in E^\rightarrow_o} \sqrt{\mathrm{w}(e)} \bE_{\mu}\left[\mathbbm{1}(\sS_{e,n}') \sqrt{\frac{\bp_e}{1-\bp_e}}\right] 
  \leq   \frac{200 A}{1-2\theta} n^{-(1+2\theta)/2}
\end{align}
for every $n\geq 1$ and every $\Gamma$-good weight function $\mathrm{w}:E\to [0,1]$. 


We now optimize over the choice of good weight function $\mathrm{w}$ in order to prove the claimed inequality~\eqref{eq:improved_two_ghost}. Fix $n\geq 1$. This inequality is trivial if $\bE_{\mu}[\mathbbm{1}(\sS_{e,n}') \sqrt{\bp_e/(1-\bp_e)}] =0$ for every $e\in E^\rightarrow_o$, so we may assume that there exists $e_0 \in E^\rightarrow_o$ for which this quantity is positive. 
Let $(A_m)_{m\geq 0}$ be an exhaustion of $E$ by finite sets containing $e_0$, so that the orbit $\Gamma A_m=\{\gamma e : \gamma \in \Gamma, e\in A_m\}$ has finite intersection with $E^\rightarrow_o$ for each $m\geq 1$. For each $m\geq 1$ we may therefore define a good weight function $\mathrm{w}_{n,m}:E\to [0,1]$ by taking \[
\mathrm{w}_{n,m}(e)= \frac{\tilde {\mathrm{w}}_{n,m}(e) \mathbbm{1}(e\in \Gamma A_m)}{\sum_{e' \in E^\rightarrow_o} \tilde {\mathrm{w}}_{n,m}(e') \mathbbm{1}(e'\in \Gamma A_m)} \; \text{ where } \; \tilde {\mathrm{w}}_{n,m}(e)= \min\left\{m,\bE_{\mu}\left[\mathbbm{1}(\sS_{e,n}') \sqrt{\frac{\bp_e}{1-\bp_e}}\right]\right\}^2  \]
for every $e\in E$.
Applying \eqref{eq:weighted_twoghost_final} with this choice of good weight function and rearranging yields that
\begin{align*}
\sum_{e\in E^\rightarrow_o \cap \Gamma A_m} \min\left\{m,\bE_{\mu}\left[\mathbbm{1}(\sS_{e,n}') \sqrt{\frac{\bp_e}{1-\bp_e}}\right]\right\}^2
  \leq   \frac{40000 A^2}{(1-2\theta)^2 n^{1+2\theta}}
\end{align*}
for every $m\geq 1$. The claim follows by taking the limit as $m\to \infty$.
\end{proof}


We now begin to work towards the proof of \cref{prop:two_ghost_weights}. Although the proof is similar to that of \cite[Theorem 3.1]{Hutchcroft2020Ising}, we will present most the details in order to keep the paper self-contained.
Let $G=(V,E,J)$ be a connected transitive weighted graph, let $\Gamma$ be a closed transitive subgroup of automorphisms of $G$, and let $\mathrm{w}:E\to[0,1]$ be a $\Gamma$-good weight function.
 For each environment $\bp\in (0,1)^E$ and subgraph $H$ of $G$, we define the $\mathrm{w}$-\textbf{fluctuation} of $H$ to be
\begin{align*}h_{\bp,\mathrm{w}}(H)&:=  \sum_{e \in E(H)} \sqrt{\mathrm{w}(e)} \left[\sqrt{\frac{\bp_e}{1-\bp_e}}\mathbbm{1}\left(e\in \partial H\right)-\sqrt{\frac{1-\bp_e}{\bp_e}} \mathbbm{1}\left(e\in E_o(H)\right) \right]\\
&=\sum_{e \in E(H)} \sqrt{\frac{\mathrm{w}(e) \bp_e}{1-\bp_e}} \cdot \frac{\bp_e-\mathbbm{1}(e\in E_o(H))}{\bp_e} \end{align*}
where $E(H)$ denotes the set of  edges that \emph{touch} $H$, i.e., have at least one endpoint in the vertex set of $H$, 
 $\partial H$ denotes the set of  edges of $G$ that touch the vertex set of $H$ but are not included in $H$, and $E_\circ(H)$ denotes the set of  edges of $G$ that are included in $H$, so that $E(H)=\partial H \cup E_o(H)$. As in \cite{Hutchcroft2020Ising}, the fluctuation is defined so that $h_{\bp,\mathrm{w}}(K_v)$ is the total quadratic variation of a certain martingale that arises when exploring the cluster $K_v$ one edge at a time after conditioning on the environment $\bp$. The following key lemma uses the mass-transport principle to relate the probability of the two-arm event to an expectation written in terms of the fluctuation. (This lemma is the only place that unimodularity is used in the proofs of any of our theorems.)

\begin{lemma}
\label{lem:AKN} Let $G=(V,E,J)$ be a connected transitive weighted graph and let $\Gamma \subseteq \Aut(G)$ be a closed transitive unimodular subgroup of automorphisms. Let $\mu$ be a $\Gamma$-invariant probability measure on $(0,1)^E$ and let $\mathrm{w}:E\to [0,1]$ be a $\Gamma$-good weight function.
Then the inequality
\begin{equation*}
\sum_{e\in E^\rightarrow_o} \sqrt{\mathrm{w}(e)}\bE_{\mu,\mathrm{w},h}\left[\mathbbm{1}(\sT_e) \sqrt{\frac{ \bp_e}{1-\bp_e}}\right] \leq 2\bE_{\mu,\mathrm{w},h}\left[\frac{|h_{\bp,\mathrm{w}}(K_{o})|}{\mathrm{w}(E(K_{o}))}  \mathbbm{1}\bigl(|K_o| < \infty \text{ and } E(K_{o}) \cap \cG \neq \emptyset \bigr)\right]
\end{equation*}
holds for every $h>0$.
\end{lemma}

Be careful to note here that $\mathrm{w}(E(K_o))$ and $h_{\bp,\mathrm{w}}(K_{o})$ are defined in terms of \emph{unoriented} edges. In particular, an edge contributes the same amount to both quantities whether it has one or two endpoints in $K_o$.

\medskip

\cref{lem:AKN} follows by a very similar proof to that of \cite[Lemma 3.3]{Hutchcroft2020Ising} but where we have allowed ourself to use the weights $\mathrm{w}$ instead of the original weights $J$. 
Before giving the proof of this lemma, let us state a variant form of the mass-transport principle involving oriented edges and good weight functions that will be useful. Let $G=(V,E)$ be a connected weighted graph, let $\Gamma \subseteq \Aut(G)$ be a unimodular closed transitive subgroup, and let $\mathrm{w} : E\to [0,1]$ be a $\Gamma$-good weight function. If $F:E^\rightarrow \times E^\rightarrow \to [0,\infty]$ is $\Gamma$-diagonally invariant in the sense that $F(\gamma e_1,\gamma e_2)=F(e_1,e_2)$ for each two oriented edges $e_1,e_2 \in E^\rightarrow$ and $\gamma \in \Gamma$, then we have that
\begin{equation}
\label{eq:edge_MTP}
\sum_{e_1 \in E_o^\rightarrow} \sum_{e_2 \in E^\rightarrow} \mathrm{w}(e_1)\mathrm{w}(e_2)F(e_1,e_2) = \sum_{e_1 \in E_o^\rightarrow} \sum_{e_2 \in E^\rightarrow} \mathrm{w}(e_1)\mathrm{w}(e_2)F(e_2,e_1).
\end{equation}
Indeed, this follows by applying the usual mass-transport principle to the function $F'(u,v) = \sum_{e_1 \in E_u^\rightarrow} \sum_{e_2 \in E_v^\rightarrow} \mathrm{w}(e_1)\mathrm{w}(e_2)F(e_1,e_2)$. The equality \eqref{eq:edge_MTP} also holds for \emph{signed} diagonally-invariant functions $F:E^\rightarrow\times E^\rightarrow \to\R$ that satisfy the absolute integrability condition
\begin{equation}
\label{eq:integrability}
\sum_{e_1\in E^\rightarrow_o} \sum_{e_2 \in E^\rightarrow} \mathrm{w}(e_1)\mathrm{w}(e_2) |F(e_1,e_2)|<\infty.
\end{equation}
This follows by applying \eqref{eq:edge_MTP} separately to the positive and negative parts of $F$, which are defined by
 $F^+(e_1,e_2)=0\vee F(e_1,e_2)$ and $F^-(e_1,e_2)=  0\vee (-F(e_1,e_2))$.

\begin{proof}[Proof of \cref{lem:AKN}]
Define $\sF_e$ to be the event that every cluster touching $e$ is finite and let $\sG_e$ be the event that there exists a finite cluster touching $e$ and $\cG$. Then $\sT_e \cap \sF_e$ is the event that the endpoints of $e$ are in distinct finite clusters each of which touches the ghost field $\cG$, and for each edge $e$ of $G$ we have that
\begin{equation*}
\mathbbm{1}(\sT_e \cap \sF_e) = \mathbbm{1}(\omega(e)=0)\cdot \#\{\text{finite clusters touching $e$ and $\cG$}\}
\\- \mathbbm{1}\bigl(\{\omega(e)=0\}\cap \sG_e).
\end{equation*}
Taking expectations conditional on the environment $\bp$, it follows that
\begin{multline}
\label{eq:unimodghost1}
\bP_{\mu,\mathrm{w},h}(\sT_e \cap \sF_e \mid \bp\,) = \bE_{\mu,\mathrm{w},h}\left[\mathbbm{1}(\omega(e)=0)\cdot\#\{\text{finite clusters touching $e$ and $\cG$}\} \mid \bp\,\right]
\\- \bP_{\mu,\mathrm{w},h}\bigl(\{\omega(e)=0\} \cap \sG_e \mid \bp\,\bigr). 
\end{multline}
Next, we observe that the event $\sF_e \cap \sG_e$ is conditionally independent of the value of $\omega(e)$ given $\bp$ and hence that
\begin{multline}
\bP_{\mu,\mathrm{w},h}\bigl(\{\omega(e)=0\} \cap \sF_e \cap \sG_e \mid \bp\, \bigr)
= \frac{1-\bp_e}{\bp_e} \bP_{\mu,\mathrm{w},h}\bigl(\{\omega(e)=1\} \cap \sF_e \cap \sG_e \mid \bp\, \bigr).
\\
= \frac{1-\bp_e}{\bp_e} \bP_{\mu,\mathrm{w},h}\bigl(\{\omega(e)=1\}  \cap \sG_e \mid \bp\,\bigr).
\label{eq:unimodghost2}
\end{multline}
Substituting \eqref{eq:unimodghost2} into \eqref{eq:unimodghost1} yields that
\begin{multline}
\label{eq:AKNmainstep}
\bP_{\mu,\mathrm{w},h}(\sT_e \cap \sF_e \mid \bp\,) = \bE_{\mu,\mathrm{w},h}\left[\mathbbm{1}(\omega(e)=0)\cdot \#\{\text{finite clusters touching $e$ and  $\cG$}\} \mid \bp\,\right]
\\- \frac{1-\bp_e}{\bp_e}\bP_{\mu,\mathrm{w},h}(\{\omega(e)=1\} \cap \sG_e \mid \bp\,) - \bP_{\mu,\mathrm{w},h}\bigl(\{\omega(e)=0\} \cap \sG_e \setminus \sF_e \mid \bp\,\bigr).
\end{multline}
Since the events $\{\omega(e)=0\} \cap \sG_e \setminus \sF_e$ and $\sT_e \cap \sF_e$ are disjoint and $\sT_e$ coincides with $(\sT_e \cap \sF_e) \cup (\{\omega(e)=0\} \cap \sG_e \setminus \sF_e)$ up to a null set, the equation \eqref{eq:AKNmainstep} implies that
\begin{multline*}
\bP_{\mu,\mathrm{w},h}(\sT_e \mid \bp\,) = \bE_{\mu,\mathrm{w},h}\left[\mathbbm{1}(\omega(e)=0)\cdot \#\{\text{finite clusters touching $e$ and  $\cG$}\} \mid \bp\right]
\\- \frac{1-\bp_e}{\bp_e}\bP_{\mu,\mathrm{w},h}(\{\omega(e)=1\} \cap \sG_e \mid \bp\,).
\end{multline*}
We can rewrite this equality more succinctly as
\begin{equation}
\label{eq:AKNmainstep2}
\bP_{\mu,\mathrm{w},h}(\sT_e \mid \bp\,)= 
 \bE_{\mu,\mathrm{w},h}\left[\frac{\bp_e-\omega(e)}{\bp_e} \cdot \#\{\text{finite clusters touching $e$ and  $\cG$}\} \;\Bigm|\; \bp\;\right].
\end{equation}
Note that this equality is essentially identical to \cite[Eq. 3.7]{Hutchcroft2020Ising}, although of course the ghost field is defined with respect to a different choice of weights there.
Consider the $\Gamma$-diagonally-invariant function $F:E^\rightarrow\times E^\rightarrow \to \R$ defined by
\begin{multline*}
F(e_1,e_2) =\\ \bE_{\mu,\mathrm{w},h}\sum\left\{\frac{1}{2\mathrm{w}(E(K))} \left[\frac{\bp_{e_1}-\omega(e_1)}{\bp_{e_1}}\right] \sqrt{\frac{\bp_{e_1}}{(1-\bp_{e_1})\mathrm{w}(e_1)}} : \begin{array}{l}\text{$K$ is a finite cluster}\\ \text{of $\omega$ touching $e_1,e_2$, and $\cG$}\end{array}\right\},
\end{multline*}
where we write $\sum\{x(i) :i\in I\} = \sum_{i\in I} x(i)$ and where we include the factor of $1/2$ to account for the fact that each edge in $E(K)$ can be oriented in two directions. (We say that an oriented edge touches $K$ if at least one of its endpoints belongs to $K$.)
The multiset of numbers being summed over has cardinality either $0,1,$ or $2$, and we can therefore compute that
\begin{align*}
 \sum_{e_1 \in E^\rightarrow_o} \sum_{e_2\in E^\rightarrow} \mathrm{w}(e_1)\mathrm{w}(e_2) |F(e_1,e_2)| 
 &\leq 2 \sum_{e_1 \in E^\rightarrow_o} \mathrm{w}(e_1) \bE_{\mu,\mathrm{w},h}\left[ \frac{|\bp_{e_1}-\omega(e_1)|}{\bp_{e_1}} \sqrt{\frac{\bp_{e_1}}{(1-\bp_{e_1})\mathrm{w}(e_1)}}\right] 
\\&= 4 \sum_{e_1 \in E_o^\rightarrow} \sqrt{\mathrm{w}(e_1)} \bE_{\mu,\mathrm{w},h} \left[\sqrt{\bp_{e_1}(1-\bp_{e_1})} \right] \leq 4 \sum_{e_1 \in E_o^\rightarrow} \sqrt{\mathrm{w}(e_1)},
\end{align*}
which is finite since $\mathrm{w}$ is good. This gives us the integrability required to apply the mass-transport principle \eqref{eq:edge_MTP} to the right hand side of \eqref{eq:AKNmainstep2} and deduce that
\begin{align}
&\sum_{e_1\in E_o^\rightarrow}\sqrt{\mathrm{w}(e_1)}\bE_{\mu,\mathrm{w},h}\left[\mathbbm{1}(\sT_{e_1})\sqrt{\frac{\bp_{e_1}}{(1-\bp_{e_1})}}\right] \nonumber
\\&\hspace{1cm}=  \sum_{e_1\in E_o^\rightarrow} \mathrm{w}(e_1)\bE_{\mu,\mathrm{w},h}\left[ \frac{\bp_{e_1}-\omega({e_1})}{\bp_{e_1}}  \sqrt{\frac{\bp_{e_1} }{(1-\bp_{e_1})\mathrm{w}(e_1)}}\cdot\#\{\text{finite clusters touching $e$ and  $\cG$}\} \right]\nonumber
\\&\hspace{1cm}= \sum_{e_1\in E_o^\rightarrow} \sum_{e_2\in E^\rightarrow} \mathrm{w}(e_1) \mathrm{w}(e_2) F(e_1,e_2) = \sum_{e_1\in E_o^\rightarrow} \sum_{e_2\in E^\rightarrow} \mathrm{w}(e_1)\mathrm{w}(e_2) F(e_2,e_1)\nonumber   \\&\hspace{1cm}=
\sum_{e_1 \in E_o^\rightarrow}\mathrm{w}(e_1)\bE_{\mu,\mathrm{w},h} \sum\left\{\frac{h_{\bp,\mathrm{w}}(K)}{\mathrm{w}(E(K))} : \begin{array}{l}\text{$K$ is a finite cluster}\\ \text{of $\omega$ touching $e_1$ and $\cG$}\end{array}\right\}.
\end{align}
Letting
$\sO_v$  be the event that the cluster $K_v$ is finite and touches $\cG$ for each vertex $v$ of $G$, we deduce that
\begin{align*}
&\sum_{e_1\in E_o^\rightarrow}\sqrt{\mathrm{w}(e_1)}\bE_{\mu,\mathrm{w},h}\left[\mathbbm{1}(\sT_{e_1})\sqrt{\frac{\bp_{e_1}}{(1-\bp_{e_1})}}\right]  
\\&\hspace{0.6cm}\leq \sum_{e_1\in E_o^\rightarrow} \mathrm{w}(e_1) \bE_{\mu,\mathrm{w},h} \sum\left\{\frac{|h_{\bp,\mathrm{w}}(K)|}{\mathrm{w}(E(K))} : \begin{array}{l}\text{$K$ is a finite cluster}\\ \text{of $\omega$ touching $e_1$ and $\cG$}\end{array}\right\}\\
&\hspace{0.6cm}\leq\sum_{e_1\in E_o^\rightarrow} \mathrm{w}(e_1) \bE_{\mu,\mathrm{w},h}\left[\frac{|h_{\bp,\mathrm{w}}(K_{o})|}{\mathrm{w}(E(K_{o}))}\mathbbm{1}\bigl(\sO_{o} \bigr)+ \frac{|h_{\bp,\mathrm{w}}(K_{e^+})|}{\mathrm{w}(E(K_{e^+}))}\mathbbm{1}\bigl(\sO_{e^+} \bigr)\right]
=2\bE_{\mu,\mathrm{w},h}\left[\frac{|h_{\bp,\mathrm{w}}(K_{o})|}{\mathrm{w}(E(K_{o}))}\mathbbm{1}\bigl(\sO_{o} \bigr)\right]
\end{align*}
as claimed, where the final equality follows by transitivity since $\sum_{e_1\in E_o^\rightarrow}\mathrm{w}(e_1)=1$.
\end{proof}


We now bound the right hand side of the inequality of \cref{lem:AKN} via a martingale analysis, where we use the assumption $\bP_\mu(|K_o|\geq n)\leq A n^{-a}$ to improve upon the analysis of \cite[Section 3.1]{Hutchcroft2020Ising}.
Let $X=(X_n)_{n\geq0}$ be a real-valued martingale with respect to the filtration $\cF=(\cF_n)_{n\geq 0}$, and suppose that $X_0=0$. The \textbf{quadratic variation process} $Q=(Q_n)_{n\geq 0}$ associated to $(X,\cF)$ is defined by $Q_0=0$ and
\[Q_n = \sum_{i=1}^n\E\left[ |X_i - X_{i-1}|^2 \mid \cF_{i-1} \right] \]
for each $n\geq 1$. 
 The following is a minor improvement of \cite[Lemma 3.4]{Hutchcroft2020Ising}.
\begin{lemma}
\label{lem:martingale_stuff}
 Let $(X_n)_{n\geq0}$ be a martingale with respect to the filtration $(\cF_n)_{n\geq 0}$ such that $X_0=0$, let  $(Q_n)_{n\geq 0}$ be the associated quadratic variation process, and let $T$ be a stopping time. Then
\[\E\Bigl[ \sup\bigl\{X_n^2 : 0 \leq n \leq T,\, Q_T \leq \lambda \bigr\} \Bigr] \leq 4 \E \left[Q_T \wedge \lambda\right] \qquad \text{for every $\lambda \geq 0$.}\]
\end{lemma}

\begin{proof}
Fix $\lambda \geq 0$ and let $\tau=\sup\{k\geq 0: Q_k \leq \lambda\}=\inf\{k\geq 0 : Q_k > \lambda\}-1$, which may be infinite. Since $Q_n$ is $\cF_{n-1}$-measurable for every $n\geq 0$, $\tau$ is a stopping time and $X_{n\wedge \tau \wedge T}$ is a martingale. 
Thus, we have by the orthogonality of martingale increments  that
\begin{align*}
\E\left[X^2_{n\wedge \tau \wedge T}\right] &= \sum_{i=1}^n\E\left[ (X_{i\wedge \tau \wedge T}-X_{(i-1)\wedge \tau \wedge T})^2\right]
= \sum_{i=1}^n\E\left[ \E\left[(X_{i\wedge \tau \wedge T}-X_{(i-1)\wedge \tau \wedge T})^2\mid \cF_{i-1} \right]\right]\\
&=\E\left[\sum_{i=1}^{n \wedge T} \E\left[(X_{i}-X_{i-1})^2\mid \cF_{i-1} \right] \mathbbm{1}(i \leq \tau)\right] = \E\left[ Q_{n \wedge \tau \wedge T}\right] \leq \E\left[ Q_{T} \wedge \lambda \right]
\end{align*}
for every $n\geq 1$. The claim follows by applying Doob's $L^2$ maximal inequality to
  $(X_{n\wedge \tau \wedge T})_{n\geq 0}$.
\end{proof}

We now apply \cref{lem:martingale_stuff} to deduce the following improvement to \cite[Lemma 3.5]{Hutchcroft2020Ising} under the assumption that the tail of the total quadratic variation satisfies a power-law upper bound.

\begin{lemma}
\label{lem:martingale_stuff2}
 Let $(X_n)_{n\geq0}$ be a martingale with respect to the filtration $(\cF_n)_{n\geq 0}$ such that $X_0=0$, and let  $(Q_n)_{n\geq 0}$ be the associated quadratic variation process. Let $T$ be a stopping time and suppose that there exist constants $A$ and $0\leq \theta <1/2$ such that
$\P( Q_T \geq x) \leq A x^{-\theta}$
for every $x > 0$. 
 Then
\begin{equation}
\label{eq:martingale_stuff2}
\E\left[ \frac{\sup_{0 \leq n \leq T}|X_n|}{Q_T} (1-e^{-h Q_T})\mathbbm{1}(0<Q_T < \infty) \right] \leq  \frac{20 A}{1-2\theta} h^{(1+2\theta)/2} \qquad \text{for every $h> 0$.} \end{equation}
\end{lemma}

\begin{proof}
Write $M_n=\max_{0\leq m \leq n} |X_n|$ for each $n\geq 0$. Since $(1-e^{-hx})/x$ is a decreasing function of $x>0$, we may write
\begin{equation*}
  \E\left[ \frac{M_T}{Q_T}\bigl(1-e^{-hQ_T}\bigr)\mathbbm{1}(0<Q_T<\infty) \right]
\leq h\sum_{k=-\infty}^\infty \frac{1-e^{-e^{k}}}{e^k} \E\left[ M_T \mathbbm{1}(e^k \leq h Q_T \leq e^{k+1})\right].
\label{eq:scales}
\end{equation*}
We can then compute that
\[
\E \left[ Q_T \wedge \lambda \right] = \int_{x=0}^\lambda \P(Q_T \geq x)\dif x  \leq \int_{x=0}^\lambda A x^{-\theta}\dif x = \frac{A}{1-\theta}\lambda^{1-\theta}
\]
for every $\lambda>0$, so that
\cref{lem:martingale_stuff} and Cauchy-Schwarz let us bound
\begin{align*}
\label{eq:martingale_Jensen_not_optimized}
\E\left[ M_T \mathbbm{1}(e^k \leq hQ_T \leq e^{k+1})\right]^2 
&\leq 4\E\left[Q_T \wedge h^{-1}e^{k+1}\right] \P\bigl(Q_T \geq h^{-1}e^k\bigr)
\nonumber
\\
&\leq \frac{4 A}{1-\theta} e^{(1-\theta)(k+1)} h^{-(1-\theta)}\cdot A e^{-\theta k} h^{\theta}
= \frac{4 A^2e^{1-\theta}}{1-\theta} e^{(1-2\theta)k} h^{2\theta-1}
\end{align*}
for each $k\in \Z$. Taking square roots and summing over $k$ we obtain that
\begin{align*}
\E\left[ \frac{M_T}{Q_T}\bigl(1-e^{-hQ_T}\bigr)\mathbbm{1}(0<Q_T<\infty) \right]
\leq \frac{2A e^{(1-\theta)/2}}{\sqrt{1-\theta}} h^{(1+2\theta)/2}\sum_{k=-\infty}^\infty \frac{1-e^{-e^{k}}}{e^{(1+2\theta)k/2}}.
\end{align*}
This series is easily seen to converge, and indeed satisfies
\begin{align*}
\sum_{k=-\infty}^\infty \frac{1-e^{-e^{k}}}{e^{(1+2\theta)k/2}} &\leq \sum_{k= 0}^\infty \frac{1}{e^{(1+2\theta)k/2}} + \sum_{k= 1}^\infty \frac{e^{-k}}{e^{-(1+2\theta)k/2}}=\frac{1}{1-e^{-(1+2\theta)/2}}+\frac{1}{e^{(1-2\theta)/2}-1}
\\&\leq \frac{\sqrt{e}+1}{(\sqrt{e}-1)(1-2\theta)}
\end{align*}
for every $0\leq \theta <1/2$, where the final inequality can be verified by calculus.
It follows that
\[
\E\left[ \frac{M_T}{Q_T}\bigl(1-e^{-hQ_T}\bigr)\mathbbm{1}(0<Q_T<\infty) \right]
\leq \frac{2 (\sqrt{e}+1) A \sqrt{2e}}{(\sqrt{e}-1)(1-2\theta)} h^{(1+2\theta)/2} \leq \frac{20 A}{1-2\theta} h^{(1+2\theta)/2}
\]
as claimed, where we used the bound $(2 (\sqrt{e}+1) \sqrt{2e}) / (\sqrt{e}-1) = 19.040\ldots \leq 20$ to simplify the constant. 
\end{proof}

\begin{proof}[Proof of \cref{prop:two_ghost_weights}]
We prove the proposition in the case that $\mu$ is supported on $(0,1)^E$, which is the only case required by our main theorems. The general case follows by a simple limiting argument that is given in detail in the proof of  \cite[Theorem 3.1]{Hutchcroft2020Ising}.
Let $\mu$ be a $\Gamma$-invariant probability measure on $(0,1)^E$, let $\mathrm{w}$ be a $\Gamma$-good weight function, and let $(\bp,\omega)$ be random variables with law $\bP_\mu$.
Write $K=K_o$ for the cluster of $o$ in $\omega$.
As in the proofs of \cite[Theorem 1.6]{1808.08940} and \cite[Theorem 3.1]{Hutchcroft2020Ising}, we can condition on the environment $\bp$ and explore the cluster $K$ one edge at a time in such a way that if $T$ denotes the (possibly infinite) total number of edges touching $K$, $E_{n}$ denotes the (random, unoriented) edge whose status is queried at the $n$th step of the exploration for each $n\geq 0$, and $\cF_{n}$ denotes the $\sigma$-algebra generated by the environment $\bp$ and the first $n$ steps of the exploration for each $n\geq 0$, then
$\bP_\mu(E_{n+1}=1 \mid \cF_n) = \bp_{E_{n+1}}$
whenever $n<T$ and $\{E_i : 1 \leq i \leq T\}=E(K)$. (Briefly, we can define such an exploration process by fixing an enumeration $E=\{e_1,e_2,\ldots\}$ and, at each step, taking $E_{n+1}$ to be minimal with respect to this enumeration among those edges that are incident to the part of the cluster of $o$ that has been explored but have not already been queried. See the above references for formal definitions.) It follows that the process $(Z_n)_{n\geq 0}$ defined by $Z_0=0$ and
\[Z_n = \sum_{i=1}^{n\wedge T} 
\sqrt{\mathrm{w}(E_i)} 
\left[\sqrt{\frac{\bp_{E_i}}{1-\bp_{E_i}}} \mathbbm{1}(\omega(E_i)=0) - \sqrt{\frac{1-\bp_{E_i}}{\bp_{E_i}}} \mathbbm{1}(\omega(E_i)=1)\right]
\]
for each $n\geq 1$ is a martingale with respect to the filtration $(\cF_n)_{n\geq 0}$ for which the final value $Z_T$ is equal to the $\mathrm{w}$-fluctuation $h_{\bp,\mathrm{w}}(K)$. Moreover, we can express the associated quadratic variation process $Q_n=\sum_{i=1}^n \bE_{\mu}[(Z_{i+1}-Z_i)^2 \mid \cF_i]$ as
\[
Q_n = \sum_{i=1}^{n\wedge T} \bE_{\mu}\left[\mathrm{w}(E_i)
\left[\frac{\bp_{E_i}}{1-\bp_{E_i}} \mathbbm{1}(\omega(E_i)=0) + \frac{1-\bp_{E_i}}{\bp_{E_i}} \mathbbm{1}(\omega(E_i)=1)\right] \Biggm| \cF_{n-1}\right] = \sum_{i=1}^{n\wedge T} \mathrm{w}(E_i)
\]
for every $n\geq 0$, so that $Q_T = \mathrm{w}(E(K))$ is the total weight of all the edges touching $K$. Thus, it follows from \cref{lem:AKN} and \cref{lem:martingale_stuff2} that if $\bP_\mu(|K|\geq n) \leq A n^{-\theta}$ for every $n\geq 1$ then
\begin{align*}
\sum_{e\in E^\rightarrow_o} \sqrt{\mathrm{w}(e)}\bE_{\mu,\mathrm{w},h}\left[\mathbbm{1}(\sT_e) \sqrt{\frac{\bp_e}{1-\bp_e}}\right] &\leq 2 \bE_{p}\left[\frac{|h_{\bp,\mathrm{w}}(K)|}{\mathrm{w}(E(K))} (1-e^{-h \mathrm{w}(E(K))} \mathbbm{1}\bigl(|K|<\infty\bigr)\right]
\nonumber
\\ &= 2 \bE_p\left[ \frac{|Z_T|}{Q_T}\bigl(1-e^{-h Q_T}\bigr)\mathbbm{1}(0<Q_T<\infty) \right] \leq \frac{40 A}{1-2\theta} h^{(1+2\theta)/2}
\end{align*}
as required.
\end{proof}

\section{Proof of the main theorem}

In this section we apply \cref{thm:hyperscalingsimple,thm:two_ghost_S} to prove \cref{thm:main,thm:general}. The proof of \cref{thm:main} relies on the following key bootstrapping lemma.

\begin{lemma}
\label{lemma:bootstrap_main}
Let $d\geq 1$, let $J:\Z^d \to (0,\infty)$ be symmetric and integrable, and suppose that there exists $\alpha < d$, $c>0$, and $r_0<\infty$ such that $J(x)\geq c \| x \|_1^{-d-\alpha}$ for every $x\in \Z^d$ with $\|x\|_1 \geq r_0$. Let $\theta=(d-\alpha)/(2d+\alpha) <1/2$. Then there exists a constant $C\geq 1$ such that the following implication holds
for each $0 \leq \beta < \beta_c$ and $1 \leq A<\infty$:
\begin{equation*}
\text{$\Bigl(\bP_\beta(|K|\geq n) \leq A n^{-\theta}$ for every $n\geq 1\Bigr)$}\\ \Rightarrow  \text{$\Bigl(\bP_\beta(|K|\geq n) \leq C A^{1/(1+\theta)} n^{-\theta}$ for every $n\geq 1\Bigr)$}.
\end{equation*}
\end{lemma}

\begin{proof}[Proof of \cref{lemma:bootstrap_main}]
By rescaling if necessary, we may assume without loss of generality that $\sum_{e\in E^\rightarrow_o} J_e =1$. Fix  $0 \leq \beta <\beta_c$ and suppose that $1 \leq A <\infty$ is such that $\bP_\beta(|K|\geq n) \leq A n^{-\theta}$ for every $n\geq 1$, where $\theta=(d-\alpha)/(2d+\alpha) <1/2$. 
We wish to prove that there exists a constant $C$ that may depend on $d,$ $\alpha$, $c$, and $r_0$ but not on the choice of $1\leq A < \infty$ or $0\leq \beta <\beta_c$ such that
\[
\bP_\beta(|K|\geq n) \leq C A^{1/(1+\theta)} n^{-\theta}
\]
for every $n\geq 1$. If $\beta \leq 1/2$ then a standard path-counting argument implies that $\bE_\beta |K_o| \leq 2$, so that the claim holds trivially  in this case by Markov's inequality provided that we take $C\geq 2$. We may therefore assume that $\beta \geq 1/2$ for the remainder of the proof.

All the constants appearing in the remainder of the proof may depend on  $d,$ $\alpha$, $c$, and $r_0$ but not on the choice of $1\leq A < \infty$ or $1/2\leq \beta <\beta_c$. 
For each $x\in \Z^d$ and $n\geq 1$, let $\sS_{x,n}'$ be the event that $0$ and $x$ belong to distinct clusters each of which contains at least $n$ vertices; both clusters are automatically finite since $\beta < \beta_c$.
Since $\theta<1/2$, we have by \cref{thm:two_ghost_S} that
there exists a constant $C_1$ such that
\[
\sum_{x \in \Z^d} (e^{\beta J_x}-1) \bP_\beta(\sS_{x,n}')^2 \leq C_1 A n^{-(1+2\theta)}
\]
for every $n \geq 1$. Let $\Lambda'_r=\Lambda_r \setminus \Lambda_{r_0-1}$ for each $r\geq r_0$. It follows by Cauchy-Schwarz that there exists a constant $C_2$ such that
\begin{align}
\hspace{-0.5em}\sum_{x \in \Lambda_r'} \bP_\beta(\sS_{x,n}') &\leq \left[\sum_{x \in \Lambda_r'} (e^{\beta J_x}-1) \bP_\beta(\sS_{x,n}')^2\right]^{1/2}\left[\sum_{x \in \Lambda_r'} \frac{1}{e^{\beta J_x}-1}\right]^{1/2} 
\nonumber
\\
&\leq C_1^{1/2} A^{1/2} n^{-(1+2\theta)/2} \left( \frac{1}{c\beta r^{-d-\alpha}} |\Lambda_r'| \right)^{1/2} 
\leq C_2 A^{1/2} n^{-(1+2\theta)/2} r^{\alpha/2} |\Lambda_r| 
\label{eq:two_ghost_massage}
\end{align}
for every $r\geq r_0$, where we used the inequality $e^x-1 \geq x$ in the first inequality on the second line.
 On the other hand, since $\theta<1/2$,
 it follows immediately from \cref{thm:hyperscalingsimple} that there exists a constant $C_3$ such that
\begin{equation}
\frac{1}{|\Lambda_r'|}\sum_{x\in \Lambda_r} \bP_\beta(0 \leftrightarrow x) \leq C_3 A^{2/(1+\theta)} |\Lambda_r'|^{-2\theta/(1+\theta)} \leq C_3 A^{2/(1+\theta)} r^{-2\theta d/(1+\theta)} 
\label{eq:hyperscaling_massage}
\end{equation}
for every $r\geq r_0$. 
We now apply these two bounds to obtain a new bound on $\bP_\beta(|K_0|\geq n)$. We have by a union bound and the Harris-FKG inequality that
\[
\bP_\beta(\sS_{x,n}') \geq \bP_\beta(|K_0|\geq n,|K_x|\geq n)-\bP_\beta(0 \leftrightarrow x) \geq \bP_\beta(|K_0|\geq n)^2-\bP_\beta(0 \leftrightarrow x).
\]
for each $x\in \Z^d$ and $n\geq 1$. Rearranging and averaging over $x\in \Lambda_r$, it follows that
\begin{align}
\bP_\beta(|K_0|\geq n)^2 &\leq 
\frac{1}{|\Lambda_r'|}\sum_{x\in \Lambda_r'}\bP_\beta(\sS_{x,n}')+
\frac{1}{|\Lambda_r'|}\sum_{x\in \Lambda_r'} \bP_\beta(0 \leftrightarrow x)
\\
&\leq  C_2 A^{1/2} r^{\alpha/2} n^{-(1+2\theta)/2}+C_3 A^{2/(1+\theta)} r^{-2d\theta/(1+\theta)}
\label{eq:nearlythere}
\end{align}
for every $r\geq r_0$ and $n\geq 1$. 
Taking 
$r=r_0 \vee \left\lceil n^{(1-2\theta)/\alpha }\right \rceil$ yields that there exists a constant $C_4$ such that
\begin{align}
\bP_\beta(|K_0|\geq n)^2 
&\leq C_4 \left( A^{1/2} n^{-2\theta}+ A^{2/(1+\theta)} n^{-2d\theta(1-2\theta)/(\alpha+\alpha \theta)}  \right)
\label{eq:nearly_done}
\end{align}
for every $n \geq 1$. Since $\theta=(d-\alpha)/(2d+\alpha)$, the two powers of $n$ appearing in this expression and equal.
Since we also have that $A^{1/2} \leq A^{2/(1+\theta)}$, it follows by taking square roots on both sides of \eqref{eq:nearly_done} that
$\bP_\beta(|K_0|\geq n)
\leq \sqrt{2C_4} A^{1/(1+\theta)} n^{-\theta}$
for every $n\geq 1$. This completes the proof.
\end{proof}

\begin{proof}[Proof of \cref{thm:main}]
We follow the same argument  used to deduce  \cref{prop:weak} from the implication \eqref{eq:simple_bootstrap_claim}; we include the details again here for ease of reading. Let $\theta=(d-\alpha)/(2d+\alpha)<1/2$. 
For each $0 \leq \beta < \beta_c$, we have by sharpness of the phase transition \cite{aizenman1987sharpness,duminil2015new} that $|K_0|$ has finite mean, and in particular that there exists $1 \leq A < \infty$ such that $\bP_\beta(|K_0|\geq n) \leq A n^{-\theta}$ for every $n\geq 1$. For each $0\leq \beta < \beta_c$ we may therefore define
\[
A_\beta = \min\bigl\{1 \leq A < \infty : \bP_\beta(|K_0|\geq n) \leq A n^{-\theta} \text{ for every $n\geq 1$}\bigr\} < \infty.
\]
Observe that the set we are minimizing over is closed, so that $\bP_\beta(|K_0|\geq n) \leq A_\beta n^{-\theta}$ for every $n\geq 1$ and $0 \leq \beta < \beta_c$.
\cref{lemma:bootstrap_main} implies that there exists a constant $C=C(d,\alpha,c,r_0)$ such that 
$A_\beta \leq C A_\beta^{1/(1+\theta)}$
for every $0 \leq \beta < \beta_c$. Since $A_\beta$ is finite for every $0\leq \beta <\beta_c$ we may safely rearrange this inequality to obtain that $A_\beta \leq C^{(1+\theta)/\theta}$ for every $0\leq \beta <\beta_c$ and hence that
\[
\bP_\beta(|K_0|\geq n) \leq C^{(1+\theta)/\theta} n^{-\theta}
\]
for every $0 \leq \beta < \beta_c$ and $n\geq 1$. This implies in particular that $\beta_c<\infty$. Considering the standard monotone coupling of $\bP_\beta$ and $\bP_{\beta_c}$ for $\beta \leq \beta_c$ and taking limits as $\beta \uparrow \beta_c$, it follows that the same estimate holds for all $0\leq \beta\leq\beta_c$ as claimed.  The claimed bound on the averaged two-point function $|\Lambda_r|^{-1}\sum_{x\in |\Lambda_r|} \bP_\beta(0\leftrightarrow x)$ follows immediately from the bound $\bP_\beta(|K_0|\geq n) \leq C^{(1+\theta)/\theta} n^{-\theta}$ together with \cref{thm:hyperscaling}.
\end{proof}

\begin{proof}[Proof of \cref{thm:general}]
This proof is very similar to that of \cref{thm:main}, and we will omit most the details.
As before, we may assume without loss of generality that $\sum_{e\in E^\rightarrow_o} J_e =1$. The analogue of \cref{lemma:bootstrap_main} is as follows: Let $\theta=(2a-1)/(a+1)<1/2$. Then there exists a constant $C$ such that the implication
\begin{multline}\text{$\Bigl(\bP_\beta(|K|\geq n) \leq A n^{-\theta}$ for every $n\geq 1\Bigr)$}\\ \Rightarrow  \text{$\Bigl(\bP_\beta(|K|\geq n) \leq C A^{1/(1+\theta)} n^{-\theta}$ for every $n\geq 1\Bigr)$}
\label{eq:general_bootstrap}
\end{multline}
holds for every $1 \leq A < \infty$ and $0 \leq \beta < \beta_c$. This will be proven via essentially the same argument as above but where we replace the set $\Lambda_r'$ with the analogous set $\Lambda_\eps=\{x \in V : \{o,x\} \in E, J_{\{o,x\}} \geq \eps \}$, which satisfies $|\Lambda_\eps| \geq c\eps^{-a}$ for every $0<\eps\leq \eps_0$ by assumption. As before, it suffices to consider the case that $\beta \geq 1/2$. Fix $1/2 \leq \beta <\beta_c$ and $1 \leq A < \infty$ and suppose that $\bP_\beta(|K|\geq n) \leq A n^{-\theta}$ for every $n\geq 1$. The derivations of \eqref{eq:two_ghost_massage} and \eqref{eq:hyperscaling_massage} from \cref{thm:two_ghost_S} and \cref{thm:hyperscalingsimple} yield in this context that there exist constants $C_1$, $C_2$, and $C_3$ such that
\begin{align}
&&\frac{1}{|\Lambda_\eps|}\sum_{x \in \Lambda_\eps} \bP_\beta(\sS'_{x,n}) &\leq C_1 A^{1/2}n^{-(1+2\theta)/2} \eps^{-(1-a)/2} 
\\
\text{and} && 
\frac{1}{|\Lambda_\eps|}\sum_{x\in \Lambda_r} \bP_\beta(0 \leftrightarrow x) &\leq C_2 A^{2/(1+\theta)} |\Lambda_\eps|^{-2\theta/(1+\theta)} \leq C_3 A^{2/(1+\theta)} \eps^{2a\theta /(1+\theta)}
\end{align}
for every $0<\eps \leq \eps_0$ and $n\geq 1$.
The same union bound and Harris-FKG argument used to derive \eqref{eq:nearlythere} then yields that
\[
\bP_\beta(|K_o| \geq n)^2 \leq C_1 A^{1/2}n^{-(1+2\theta)/2} \eps^{-1/2} + C_3 A^{2/(1+\theta)} \eps^{2a\theta/(1+\theta)}
\]
for every $0<\eps \leq \eps_0$ and $n\geq 1$. Taking $\eps=\eps_0 \wedge n^{-(1-2\theta)/(1-a)}$ implies that there exists a constant $C_4$ such that
\[
\bP_\beta(|K_o| \geq n)^2 \leq C_4 A^{1/2} n^{-2\theta} + C_4 A^{2/(1+\theta)} n^{-2a\theta (1-2\theta)/((1-a)(1+\theta))}
\]
for every $n\geq 1$. As before, the definition of $\theta$ is chosen such that these two powers of $n$ are equal, and we obtain that 
$\bP_\beta(|K_o|\geq n) \leq \sqrt{2C_4} A^{1/(1+\theta)} n^{-\theta}$
for every $n\geq 1$.  This completes the proof of the implication \eqref{eq:general_bootstrap}. The derivation of \cref{thm:general} from the implication \eqref{eq:general_bootstrap} is identical to the derivation of \cref{thm:main} from \cref{lemma:bootstrap_main} and is omitted.
\end{proof}


\paragraph{Acknowledgments.}
We thank Jonathan Hermon for his careful reading of an earlier version of this manuscript, and thank Gordon Slade for helpful discussions on the physics literature. We also thank the anonymous referee for their helpful comments and corrections.

 \setstretch{1}
 \footnotesize{
  \bibliographystyle{abbrv}
  \bibliography{unimodularthesis.bib}
  }

   \emph{Data sharing not applicable to this article as no datasets were generated or analysed during the current study.}

\end{document}

%% file: header.tex
\usepackage{amsmath,amssymb,amsfonts,amsthm}
\usepackage{color}

\usepackage[colorlinks=true,linkcolor=blue,citecolor=blue,urlcolor=blue,pdfborder={0 0 0}]{hyperref}
\usepackage{cleveref}

\usepackage{caption}
\usepackage{thmtools}

\usepackage{mathrsfs}

\crefname{theorem}{Theorem}{Theorems}
\crefname{thm}{Theorem}{Theorems}
\crefname{lemma}{Lemma}{Lemmas}
\crefname{lem}{Lemma}{Lemmas}
\crefname{remark}{Remark}{Remarks}
\crefname{prop}{Proposition}{Propositions}
\crefname{defn}{Definition}{Definitions}
\crefname{corollary}{Corollary}{Corollaries}
\crefname{conjecture}{Conjecture}{Conjectures}
\crefname{question}{Question}{Questions}
\crefname{chapter}{Chapter}{Chapters}
\crefname{section}{Section}{Sections}
\crefname{figure}{Figure}{Figures}
\crefname{example}{Example}{Examples}

\theoremstyle{plain}
\newtheorem{thm}{Theorem}[section]

\newtheorem{lemma}[thm]{Lemma}
\newtheorem{theorem}[thm]{Theorem}

\newtheorem{corollary}[thm]{Corollary}

\newtheorem{prop}[thm]{Proposition}

\theoremstyle{definition}

\theoremstyle{remark}
\newtheorem{remark}[thm]{Remark}

\numberwithin{equation}{section}

%% file: commands.tex
\renewcommand{\P}{\mathbb P}
\newcommand{\E}{\mathbb E}

\newcommand{\R}{\mathbb R}
\newcommand{\Z}{\mathbb Z}

\newcommand{\cE}{\mathcal E}
\newcommand{\cF}{\mathcal F}
\newcommand{\cG}{\mathcal G}

\newcommand{\sD}{\mathscr D}
\newcommand{\sE}{\mathscr E}
\newcommand{\sF}{\mathscr F}
\newcommand{\sG}{\mathscr G}

\newcommand{\sO}{\mathscr O}

\newcommand{\sT}{\mathscr T}

